\documentclass[11pt]{article}
\usepackage{amsmath, amssymb, amsthm, wasysym,todonotes, graphicx}
\usepackage[margin=3cm]{geometry}
\usepackage{color}
\usepackage{hyperref}
\usepackage{caption}
\usepackage{subcaption}
\usepackage{url}
\usepackage{titlesec}
\usepackage{floatrow}
\usepackage{chngcntr}
\usepackage{manfnt}
\usepackage{verbatim}
\usepackage{comment}
\usepackage{mathrsfs}

\usepackage{authblk}

\counterwithin{figure}{section}

\titleformat{\section}{\large\bfseries}{\thesection.}{1em}{}

\theoremstyle{definition}
\theoremstyle{definition}
\newtheorem{definition}{Definition}[section]
\newtheorem{theorem}{Theorem}[section]
\newtheorem{question}{Question}[section]
\newtheorem{proposition}{Proposition}[section]

\newtheorem{lemma}{Lemma}[section]

\theoremstyle{remark}
\newtheorem*{remark}{Remark}

\newcommand{\R}{\mathbb{R}}

\newcommand{\A}{\mathcal{A}}
\newcommand{\V}{\mathcal{V}}

\newcommand{\T}{\mathcal{T}}
\newcommand{\ttmin}{T_{\min}}
\newcommand{\ttmax}{T_{\max}}
\DeclareMathOperator{\conv}{conv}
\DeclareMathOperator{\vertex}{vert}
\newcommand{\sig}[1]{\widehat{\Sigma}_{#1}}
\DeclareMathOperator{\lift}{lift}
\DeclareMathOperator{\level}{level}
\DeclareMathOperator{\volu}{vol}
\newcommand{\ts}{\textsuperscript}

\newcommand{\hv}{\mathcal{H}_{\mathcal{V}}}
\newcommand{\shift}{\overset{\mathrm{shift}}{=\joinrel=}}
\newcommand{\tmax}{\mathcal{T}_{\max}}
\newcommand{\tmin}{\mathcal{T}_{\min}}
\newcommand{\gkz}{\Sigma^{\text{GKZ}}_{\mathcal{A}}}
\newcommand{\vertgkz}{\text{vert}^{\text{GKZ}}(\tau)}
\newcommand*\cube{\mbox{\mancube}}
\newcommand{\zono}{\mathcal{Z}_{\mathcal{V}}} %omg that's exactly what i was going to call it ////girl i think u mean ZOMG

\title{Higher Secondary Polytopes for Two-Dimensional Zonotopes}

\author[1]{Elisabeth Bullock \thanks{edb22@mit.edu}}
\author[2]{Katie Gravel \thanks{kgravel@mit.edu}}
\affil[1,2]{Department of Mathematics, Massachusetts Institute of Technology}
\date{}
\begin{document}

%\affil[2]{Department of Mathematics, Massachusetts Institute of Technology}
\maketitle

\begin{abstract}
\iffalse
Very recently, Galashin, Postnikov, and Williams introduced the notion of \emph{higher secondary polytopes}, generalizing the \emph{secondary polytope} of Gelfand, Kapranov, and Zelevinsky. Given an $n$-point configuration $\mathcal{A}$ in $\mathbb{R}^{d-1}$, they define a family of convex $(n-d)$-dimensional polytopes $\widehat{\Sigma}_{1}, \ldots, \widehat{\Sigma}_{n-d}$. The $1$-skeletons of this family of polytopes are the flip graphs of certain combinatorial configurations which generalize triangulations of $\text{conv} \mathcal{A}$. In our paper we restrict our attention to $d=2$. %First, we show that $\sig{1}$ is the product of simplices. 
First, we relate the $1$-skeleton of the Minkowski sum $\widehat{\Sigma}_{k} + \widehat{\Sigma}_{k-1}$ to the flip graph of ``hypertriangulations'' of the deleted $k$-sum of $\A$ when $\A$ consists of distinct points. Second, we compute the diameter of $\widehat{\Sigma}_{k}$ and $\widehat{\Sigma}_{k}+\widehat{\Sigma}_{k-1}$ for all $k$.
\fi

Very recently, Galashin, Postnikov, and Williams introduced the notion of \emph{higher secondary polytopes}, generalizing the \emph{secondary polytope} of Gelfand, Kapranov, and Zelevinsky. Given an $n$-point configuration $\mathcal{A}$ in $\mathbb{R}^{d-1}$, they define a family of convex $(n-d)$-dimensional polytopes $\widehat{\Sigma}_{1}, \ldots, \widehat{\Sigma}_{n-d}$. The $1$-skeletons of this family of polytopes are the flip graphs of certain combinatorial configurations which generalize triangulations of $\text{conv} \mathcal{A}$. In our paper we restrict our attention to $d=2$. 
First, we relate the $1$-skeleton of the Minkowski sum $\widehat{\Sigma}_{k} + \widehat{\Sigma}_{k-1}$ to the flip graph of ``hypertriangulations'' of the deleted $k$-sum of $\mathcal{A}$ when $\mathcal{A}$ consists of distinct points. Second, we compute the diameter of $\widehat{\Sigma}_{k}$ and $\widehat{\Sigma}_{k}+\widehat{\Sigma}_{k-1}$ for all $k$.
\end{abstract}

%\tableofcontents

%\newpage

\section{Introduction}

A number of interesting questions regard the set of all triangulations of a given point cloud, together with the \emph{flips} between those triangulations. If one considers all the triangulations of a convex $n$-gon, with the flip operation switching the diagonal of a quadrilateral, the corresponding \emph{flip graph} turns out to be the 1-skeleton of a convex $(n-3)$-dimensional polytope, known as the Stasheff associahedron~\cite{stasheff1963homotopy}. Computing the diameter of this graph is notoriously hard; it was done partially by Sleator–-Tarjan–-Thurston~\cite{sleator1988rotation} and finalized by Pournin~\cite{pournin2014diameter}.

A more general setting describing triangulations in any dimension is given in the work of Gelfand, Kapranov, and Zelevinsky~\cite{gelfand1994discriminants}. For a given point set, they introduce its \emph{secondary polytope} which generalizes the associahedron and reflects a good deal of combinatorial information about the triangulations.

There are a few natural ways to generalize the notion of a triangulation, starting from the following perspective: a triangulation of an $n$-point set $\A$, lying in an affine $(d-1)$-plane in $\R^d$, can be viewed as the bijective image of a $(d-1)$-dimensional subcomplex inside the standard $(n-1)$-simplex $\triangle_{1,n} \subset \R^n$, under the linear projection $\pi : \triangle_n \to \R^d$ mapping the vertices of $\triangle_n$ to the points of $\A$. If we extend the map $\pi$ to the unit cube $\cube^{\, n} = [0,1]^n$, so that $\triangle_{1,n} = \cube^{\, n} \cap \{(x_1,\ldots,x_n)~|~x_1+\ldots+x_n=1\}$, we might as well consider the (fine) zonotopal tilings of the zonotope $\pi(\cube^{\, n})$, which restrict to triangulations of $\pi(\triangle_{1,n})$. Now if we decide to look at the polytope $\pi(\triangle_{k,n})$, where $\triangle_{k,n} = \cube^{\, n} \cap \{(x_1,\ldots,x_n)~|~x_1+\ldots+x_n=k\}$ is the standard hypersimplex, there are a few remarkable ways to subdivide it.

\begin{itemize}
    \item  One can consider the bijective images of $(d-1)$-dimensional subcomplexes inside $\triangle_{k,n}$; these subdivisions are called \emph{hypertriangulations}~\cite{olarte2019hypersimplicial}.
    \item One can consider the tilings of $\pi(\triangle_{k,n})$ induced by the zonotopal tilings of $\pi(\cube^{\, n})$; these are \emph{lifting hypertriangulations};
    \item One can as well consider the lifting hypertriangulations up to certain zonotopal flips; in this paper we call those \emph{reduced lifting hypertriangulations}; they are closely related to the vertices of the \emph{higher secondary polytopes}~\cite{galashin2019higher}.
\end{itemize}

Each of the three ways have their notions of a flip, so we can consider the corresponding flip graphs. In this project, we are mainly concerned with the second and the third way of defining generalized triangulations in the case when $d=2$. The main results of this paper regard the diameter of those flip graphs. 

Galashin, Postnikov, and Williams introduced~\cite{galashin2019higher} \emph{higher secondary polytopes}, $\sig{1}, \ldots, \sig{n-d}$, a family of convex $(n-d)$-dimensional polytopes for a $n$-point configuration $\mathcal{A}$ in $\mathbb{R}^{d-1}$.  The first of these polytopes coincides with the secondary polytope of Gelfand, Kapranov, and Zelevinsky. The Minkowski sum $\sig{1}+\ldots+\sig{n-d}$ is closely related to the fiber zonotope of Billera and Sturmfels \cite{billera1992fiber}, and its 1-skeleton gives the flip graph for \emph{regular zonotopal tilings} related to $\A$. Individual higher secondary polytopes capture information about certain equivalence classes of regular zonotopal tilings and generalize the GKZ secondary polytope.

%dimension 1 and dimension 3
Higher secondary polytopes are relatively well-understood when $d=1$ or $d=3$.  In the one-dimensional case, $\sig{k}$ is the hypersimplex $\triangle_{k,n}$, while the Minkowski sum $\sig{1}+...+\sig{n-1}$ gives the well-known $(n-1)$-dimensional permutohedron.
The three-dimensional case is closely connected to \emph{plabic graphs}, certain planar bi-color graphs, introduced by Postnikov \cite{postnikov2006total} in his study of totally positive Grassmannian. For example, for a set of $n$ points in convex position in $\mathbb{R}^2$, the $1$-skeleton of the polytope $\sig{k}$ gives the flip graph for certain bipartite plabic graphs.

%diameter problem
Since the $1$-skeletons of these higher secondary polytopes (and certain Minkowski sums of higher secondary polytopes) yield flip graphs for certain combinatorial objects, it is natural to investigate their diameters  (that is to say, the diameters of their $1$-skeletons, measured as the graph edge-distance).  It is easy to compute the diameter of these polytopes when $d=1$, but the $d=3$ case has proven fairly challenging.  When $d=3$, the diameter of $\sig{1}$ is $2n-10$ for all $n>12$ \cite{sleator1988rotation}\cite{pournin2014diameter}.  When $d=3$ and $n=2k$, Farber conjectured that the diameter of $\sig{k}$ is $\frac{k}{2}(k-1)^2$. Balitskiy and Wellman \cite{balitskiy2018flips} obtained this value as a lower bound, but a non-trivial upper bound has not been achieved.

%our own thing /// o hell ya
Our research was motivated by the desire to better understand the case $d=2$, with a specific interest in the diameters of the higher secondary polytopes. %We first show that $\sig{1}$ is combinatorially equivalent to the Cartesian product of certain simplices, 
We restrict our attention to \emph{generic} point configurations (ones in which each point has multiplicity one), and compute the diameters of all higher secondary polytopes as well as the diameters of $\sig{k}+\sig{k-1}$ for all $k\in[n-2]$, which holds combinatorial significance in relation to lifting hypertriangulations.

\subsection*{Acknowledgements} This work was done as part of the Summer Program in Undergraduate Research (SPUR) at MIT. The authors thank Pavel Galashin for suggesting this project and providing them with useful advice. They are grateful to Professor Ankur Moitra and Professor David Jerison for their helpful and encouraging comments. Lastly, they are grateful for their mentor, Alexey Balitskiy, for offering many insightful ideas and for being consistently supportive throughout this entire project.

\section{Background}

\subsection{Triangulations and the Gelfand--Kapranov--Zelevinsky secondary polytope} %talk ab regularity and gkz stuff

Given a point configuration $\mathcal{A}=\{a_1,...,a_n\}$ in $\mathbb{R}^{d-1}$, we consider \emph{triangulations} of $\mathcal{A}$, or polyhedral subdivisions of $\conv(\mathcal{A})$ in which every cell is a $(d-1)$-dimensional simplex whose vertices belong to $\mathcal{A}$. We label the cells by their vertices: $\triangle_B=\conv\left(\{a_b~|~b\in B\}\right)$ where $B\in\binom{[n]}{d}$. Note that it is not required for every point in $\A$ to be a vertex of a simplex in this subdivision.

More specifically we can consider \emph{regular} triangulations of $\mathcal{A}$.
\begin{definition}
Given a height vector $h=(h_1,...,h_n) \in \R^n$, each point $a_i\in\mathcal{A}$ is lifted to a the point $\Tilde{a}_i=(a_i,h_i)$ in $\mathbb{R}^d$. An \emph{upper face} of a polytope $P$ is a face $F$ for which $x+\epsilon e_d\notin P$ for all $x\in F$ and $\epsilon>0$, where $e_i$ is the $i\ts{th}$ standard basis vector. We obtain a polyhedral subdivision of $\A$ when the upper faces of $\conv(\Tilde{a}_i,...,\Tilde{a}_n)$ are projected back down to $\mathbb{R}^{d-1}$ by forgetting the last coordinate.  The resulting subdivison is called \emph{regular}.
\end{definition}

We can also consider \emph{flips} between triangulations.  Informally,  $d+1$ points in general position in $\mathbb{R}^{d-1}$ can be triangulated in exactly two ways, and switching between them is called a flip (for a complete definition see \cite{pfeifle2003computing}). For example, when $n=5$ and $d=3$, the flip graph of triangulations is isomorphic to the cycle graph on $5$ vertices.

\begin{comment}

\begin{figure}[h]
    \centering
    \includegraphics[width=0.4\textwidth]{Pentagon.PNG}
    \caption{The flip graph for triangulations of a pentagon}
    \label{fig:pentagon}
\end{figure}
\end{comment}

Information about regular triangulations of $\mathcal{A}$ is captured by the \emph{secondary polytope} $\gkz$, introduced by Gelfand--Kapranov--Zelevinsky in \cite{gelfand1994discriminants}.  They associated a point to each triangulation of $\mathcal{A}$:
$$
\vertgkz:=\sum_{\triangle_B\in\tau}\left(\volu^{d-1}(\triangle_B)\cdot\sum_{b\in B} e_b\right),
$$

where $\volu^{d-1}$ is the standard Euclidean volume in $\mathbb{R}^{d-1}$.

\begin{definition} The \emph{secondary polytope} of $\A$ is defined as
$$
\gkz:=\conv({\vertgkz}~|~\tau\text{ is a triangulation of }\mathcal{A})
$$
\end{definition}
In \cite{gelfand1994discriminants} it is shown that the 1-skeleton of $\gkz$ is the flip graph  of regular triangulations of $\mathcal{A}$.

In the case $d-1=2$, if the points of $\A$ are in (strictly) convex position, it turns out that all the triangulations are regular, and the corresponding secondary polytope is the \emph{Tamari--Stasheff associahedron}.
\iffalse
\begin{figure}[h]
    \centering
    \includegraphics[width=0.5\textwidth]{associahedron.png}
    \caption{The associahedron for $|\A|=6$}
    \label{fig:associahedron}
\end{figure}
\fi

\subsection{Fine zonotopal tilings}\label{sect:fine_tilings}
 A $d$-dimensional zonotope is the Minkowski sum of line segments in $\R^d$. Given any point configuration $\mathcal{A}=\{a_1,...,a_n\}$ in $\mathbb{R}^{d-1}$, we can \emph{lift} to an associated zonotope in $\R^d$. Let $\mathcal{V}=\{(a_i,1) \in \R^d\mid a_i\in\mathcal{A}\}$. Then, the Minkowski sum of $v\in\V$ is the zonotope of $\mathcal{V}$ which we denote $\zono$.

Let $\cube^{\, n}$ be the unit cube in dimension $n$, and let $\pi\colon \cube^{\, n}\to\R^d$ be the linear projection defined by $\pi(e_i)=v_i$ for $i\in[n]$. Let a face of $\cube^{\, n}$ $F_{A,B}$ be defined by the equations: $x_i=1$ if $i\in A$, and $x_i=0$ if $i\in [n]\setminus (A\sqcup B)$. 
\begin{definition}
 A \emph{fine zonotopal tiling} $\T$ of $\zono$ is a collection of $d$-dimensional faces $F_{A,B}$ of the $n$-cube such that

\begin{itemize}
    \item the images of $\Pi_{A,B}=\pi(F_{A,B})$ for all $F_{A,B}\in\T$, are $d-$dimensional parellelpipeds that form a polyhedral subdivsion of $\zono$.   
    \item For any two faces $F_{A_1,B_1}$ and $F_{A_2,B_2}$, $\pi(F_{A_1,B_1}\cap F_{A_2,B_2})=\Pi_{A_1,B_1}\cap\Pi_{A_2,B_2}$.
\end{itemize}
\end{definition}

We denote the tiles in a fine zonotopal tiling by $$\Pi_{A,B}=\sum_{a\in A}e_a+\sum_{b\in B}[0,v_b].$$

\begin{definition}\label{def:reg-tiling}
A fine \emph{regular} zonotopal tiling is a fine zonotopal tiling that can be constructed as follows:  there is a height vector $h\in\R^n$ such that for all $v\in\mathcal{V},$ $\zono$ can be obtained from projecting the upper boundary of the lifted zonotope obtained by lifting $v_i$ to $(v_i,h_i)$.
\end{definition}

In other literature, regular tilings are also referred to as \emph{coherent} tilings.
\begin{figure}[h]
    \centering
    \includegraphics[width=0.6\textwidth]{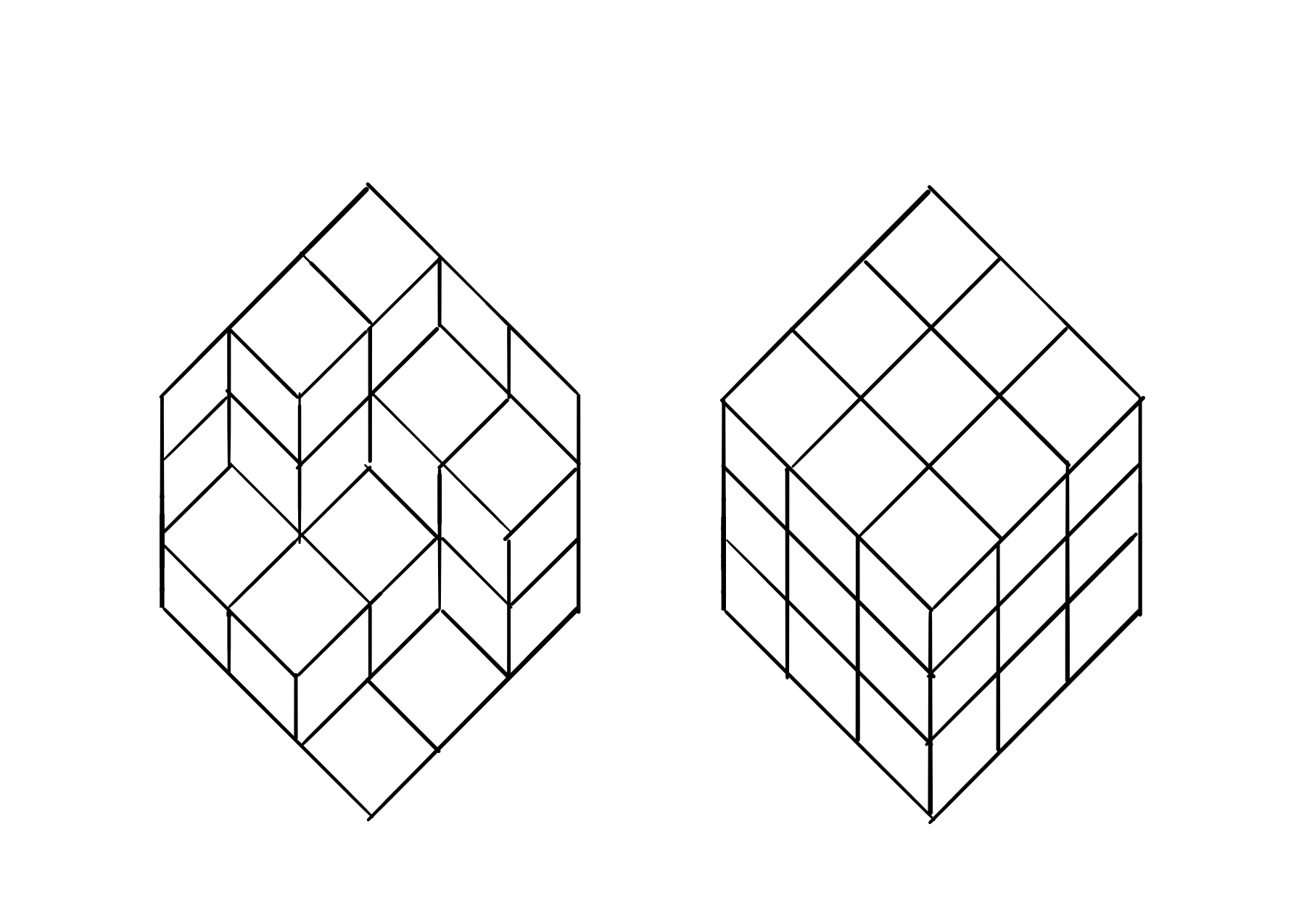}
    \caption{An irregular zonotopal tiling (left) and a regular zonotopal tiling (right)}
    \label{fig:irregular}
\end{figure}

\subsection{Flips in fine zonotopal tilings} %talk ab circuit stuff and i guess paths in R^n corresponding to sequences of flips

A \textit{flip} $F$ is an elementary transformation between two zonotopal tilings for a given point configuration $\mathcal{A}$.  If $|\mathcal{A}|=n=d+1$, there are exactly two fine zonotopal tilings of the associated zonotope $\zono$.  A flip equates to finding a translated copy of a zonotope on $d+1$ vectors spanning $\mathbb{R}^d$ and switching to the alternate tiling.

It is known that all regular fine zonotopal tilings for a given point configuration are flip-connected \cite{galashin2019higher}; further, when $d=2$, it is known that \textit{all} fine zonotopal tilings for a \emph{generic} point configuration are flip connected.

We can describe flips precisely in the context of \emph{circuits}:

\begin{definition}
A \textit{circuit} $C=(C^+,C^-)$ is a pair of disjoint subsets $C^+,C^-\subseteq[n]$ such that
\begin{enumerate}
    \item $(C^+\cup C^-) \setminus \{i\}$ is an independent set for all $i\in C^+\cup C^-$
    \item There exists a vector $\alpha(C)$ satisfying
    \begin{enumerate}
        \item $\alpha_i(C)>0$ for all $i\in C^+$
        \item $\alpha_i(C)<0$ for all $i\in C^-$
        \item $\alpha_i(C)=0$ for all $i\notin C^+\cup C^-$
        \item $\sum_{i\in [n]}\alpha_i(C) v_i=0$
    \end{enumerate}
\end{enumerate}
\end{definition}

For simplicity, we write $\underline{C}=C^+\cup C^-$ and $-C=(C^-,C^+)$.

Let $\mathcal{C_V}$ denote the set of circuits for a given vector configuration $\mathcal{V}$.  In the case $d=2$, circuits take one of two forms: $(\{i\},\{j\})$, where $v_i=v_j$ (here $\alpha(C)=e_i-e_j$), or a $\underline{C}=\{i,j,k\}$, where $v_i,v_j$ and $v_k$ are pairwise distinct.

We say a set $S\in [n]$ \textit{orients} $C$ positively if $C^+\subseteq S$ and $C^-\cap S=\emptyset$, and it orients $C$ positively if $C^-\subseteq S$ and $C^+\cap S=\emptyset$.  We say a collection of sets in $[n]$ orients $C$ positively if some set in the collection orients $C$ positively but no set orients $C$ negatively, and the collection orients $C$ negatively if some set in the collection orients $C$ negatively but no set orients $C$ positively.

We say a set $S\in[n]$ is a \emph{vertex} of the tiling $\T$ if $A\subseteq S\subseteq A\cup B$ for some tile $\Pi_{A,B}\in \T$. Let $V(\T)$ be the set of vertices of the tiles in $\T$.

\begin{proposition}
(\cite[Theorem 2.7 and Corollary 7.22]{galashin2017purity}) Given a fine zonotopal tiling $\T$ of $\zono$ and a circuit $C\in\mathcal{C_V}$, $V(\T)$ either orients $C$ positively or negatively.
\end{proposition}

\begin{definition}
For a fine zonotopal tiling $\T$, we define a map $\sigma_{\T}: \mathcal{C_V}\rightarrow \{-1,1\}$ as follows:
\[
\sigma_{\T}(C)=
\begin{cases}
    1 &\text{if $\T$ orients $C$ positively}\\
    -1 &\text{if $\T$ orients $C$ negatively}
\end{cases}
\]
\end{definition}

We say there exists a \textit{flip}  along circuit $C$ between two tilings $\T_1$ and $\T_2$ if $\sigma_{\T_1}(C)=-\sigma_{\T_2}(C)$ and $\sigma_{\T_1}(C')=\sigma_{\T_2}(C')$ for all $C'\in\mathcal{C_V}$ such that $C'\neq\pm C$.

Using this definition, \cite{galashin2019higher} explicitly determined the change in the set of tiles between two tilings $\T_1$ and $\T_2$ that are connected by a single flip along circuit $C$.  Here, we give this description specifically for \emph{generic} point configurations (see \cite[Proposition 5.12]{galashin2019higher} for a generalization to nongeneric point configurations):

\begin{proposition}(\cite[Proposition 5.7]{galashin2019higher})
Let $F$ be a flip along circuit $C$ between $\T_1$ and $\T_2$.  Then there exists a set $A(F)\subseteq[n]\setminus\underline{C}$ such that

$$
\T_1\setminus \T_2=  \left\{ \Pi_{A(F)\cup \{j\},\underline{C}\setminus\{j\}}\right\}_{j\in C^+}\sqcup \left\{ \Pi_{A(F),\underline{C}\setminus\{j\}}\right\}_{j\in C^-} 
$$
$$
\T_2\setminus \T_1=   \left\{ \Pi_{A(F),\underline{C}\setminus\{j\}}\right\}_{j\in C^+}\sqcup \left\{ \Pi_{A(F)\cup\{j\},\underline{C}\setminus\{j\}}\right\}_{j\in C^-} 
$$

\end{proposition}

\begin{figure}[h]
    \centering
    \includegraphics[width=0.625\textwidth]{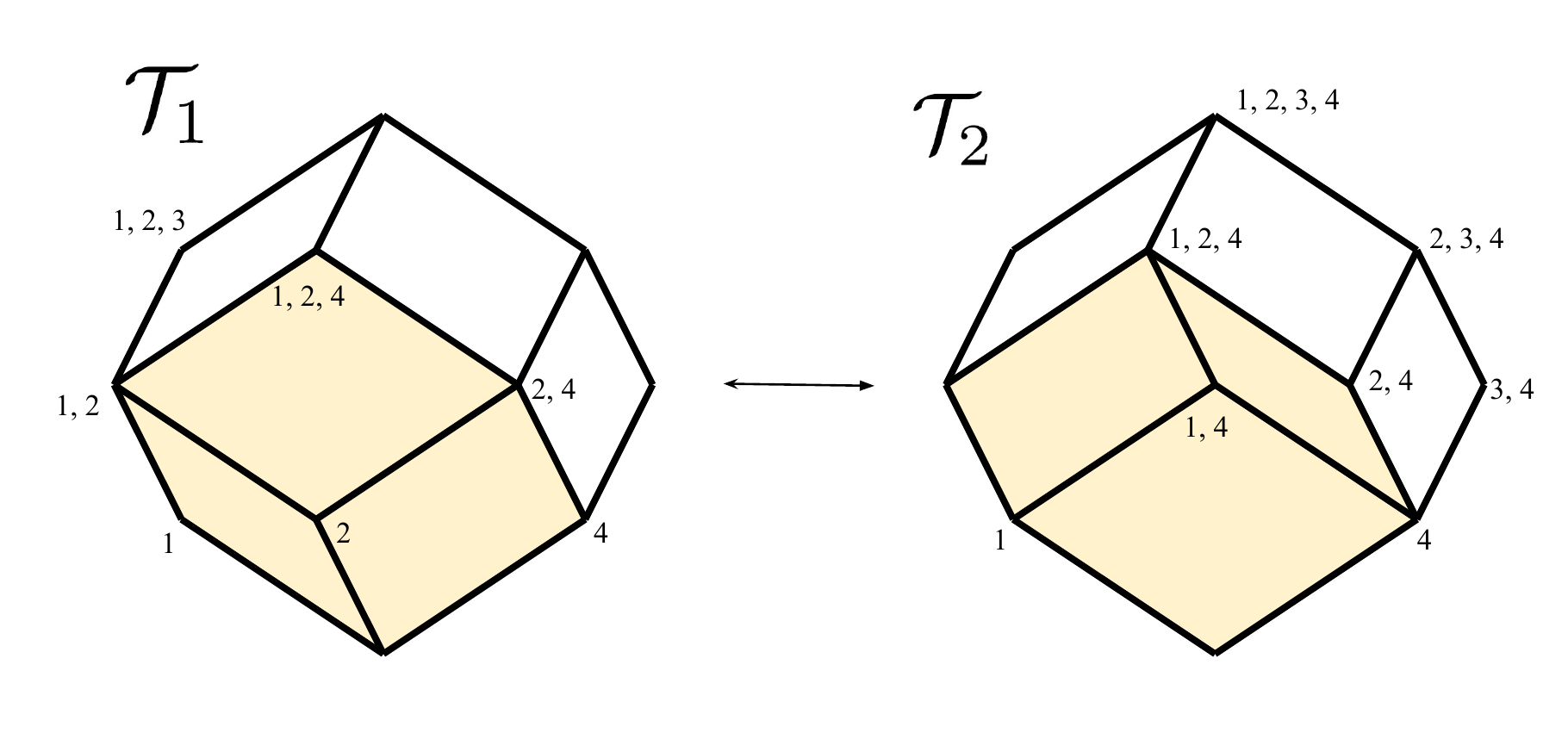}
    \caption{A flip between $\T_1$ and $\T_2$ along circuit $C=(\{1,4\},\{2\})$ for a generic point configuration $\A=\{-1,0,1,2\}$.} 
    \label{fig:generic flip}
\end{figure}

We say the \textit{level} of a flip to be $\level(F)=|A(F)|+1$.

For \emph{regular} zonotopal tilings, circuit orientations can also be described using vectors in $\mathbb{R}^n$. Let $\hv$ be the set of vectors in $\mathbb{R}^n$ such that $\langle h,\alpha(C)\rangle= 0$ for some $C\in\mathcal{C_V}$, where $\langle\cdot,\cdot\rangle$ denotes the standard dot product.
We say a height vector $h\in\mathbb{R}^n$ is \textit{generic} if $h\in\mathbb{R}^n\setminus\hv$.

\begin{definition}
For a generic height vector $h$, define the map $\sigma_h\colon\mathcal{C_V}\rightarrow\{1,-1\}$ as follows: 
\[
\sigma_h(C)=
\begin{cases}
1 & \text{if } \langle h,\alpha(C)\rangle>0\\
-1 & \text{if } \langle h,\alpha(C)\rangle<0
\end{cases}
\]
\end{definition}

For a generic height vector $h$, define $\T_h$ to be the fine zonotopal tiling such that $\sigma_{\T_h}=\sigma_h$. Galashin, Postnikov, and Williams showed that $\T_h$ is exactly the regular tiling obtained in definition \ref{def:reg-tiling} using height vector $h$ \cite{galashin2019higher}.  

Using this notion of circuits, we can describe sequences of flips between regular tilings in terms of paths through $\mathbb{R}^n$ which intersect one hyperplane in $\hv$ at a time. Note that this description of flips can easily give us the well known flip diameter for regular fine zonotopal tilings of $\binom{n}{d+1}$ in the generic case.

\subsection{Higher secondary polytopes} 
We restrict our attention to the case where $\A\subset\R^{d-1}$ contains distinct points. We define the main object of study for these configurations, which generalizes the GKZ secondary polytope.

Let $\T$ be a fine zonotopal tiling. Let $\{e_i\}_{i\in[n]}$ be the standard basis for $\R^n$. Let $e_A=\sum_{a\in A}e_a$ where $a\in[n]$. Furthermore, let $\text{Vol}^d(\Pi_{A,B})=|\det (v_i)_{i\in B})|$. Define:
$$\widehat{\text{vert}}_k(\T)=\sum_{\Pi_{A,B}\in\T,\\|A|=k}\volu^d(\Pi_{A,B})\cdot e_A. $$
\begin{definition}
The \emph{$k^\text{th}$ higher secondary polytope} is
$$\sig{\A,k}=\conv\{\widehat{\text{vert}}_k(\T)\mid \T\text{ is a fine regular zonotopal tiling}\}.$$
\end{definition}
We often omit $\A$ from the notation and write $\sig{k}$.

If $k\notin[n-d]$, $\sig{k}$ is a single point. It is conjectured in \cite{galashin2019higher} that the word \emph{regular} can be omitted from the definition of higher secondary polytope. Meaning, it is conjectured that $\widehat{\vertex}_k(\T)$ for nonregular tilings lies within the convex hull of the regular tilings. 

\begin{figure}[!htb]
\minipage{0.32\textwidth}\centering
  \includegraphics[width=5cm]{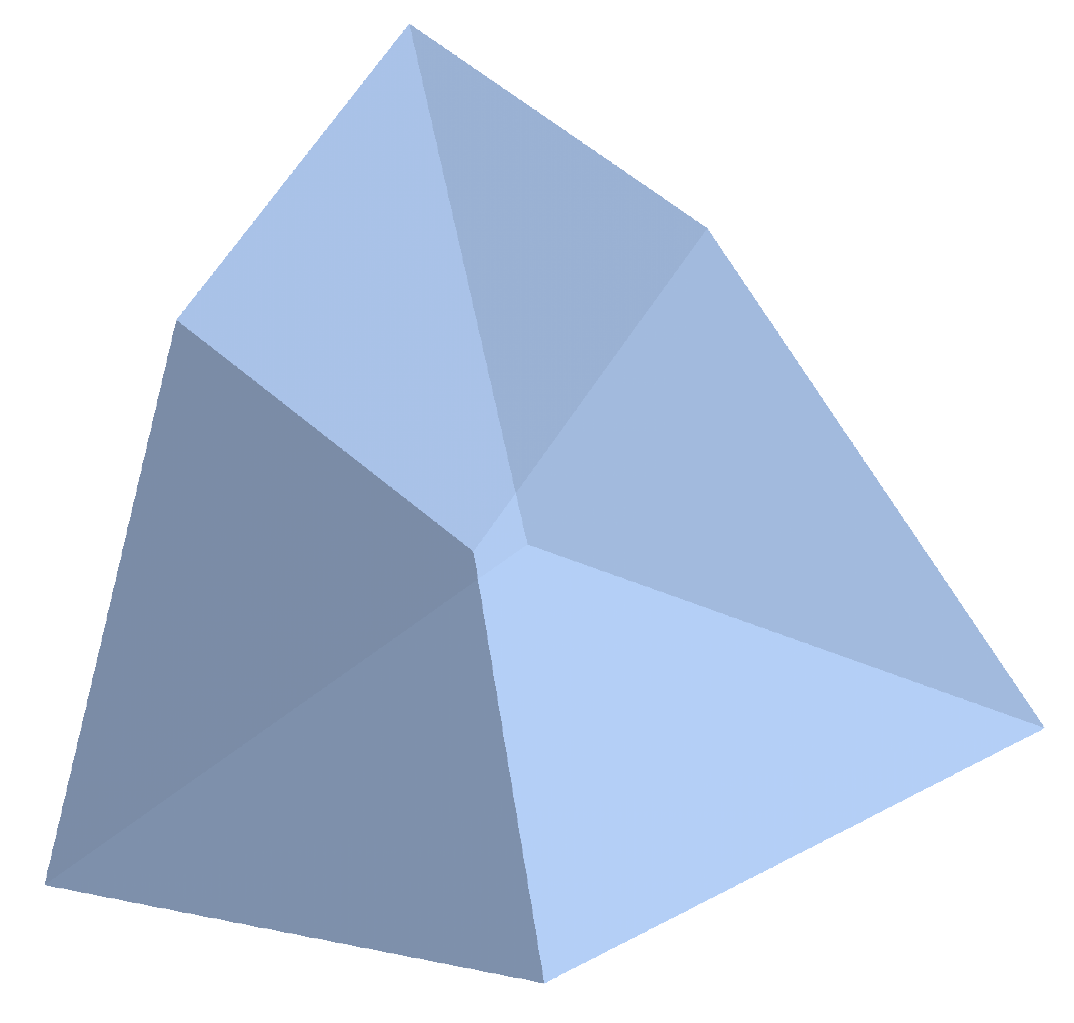}
  %\caption*{$\sig{1}$}
\endminipage\hfill
\minipage{0.32\textwidth}\centering
  \includegraphics[width=5cm]{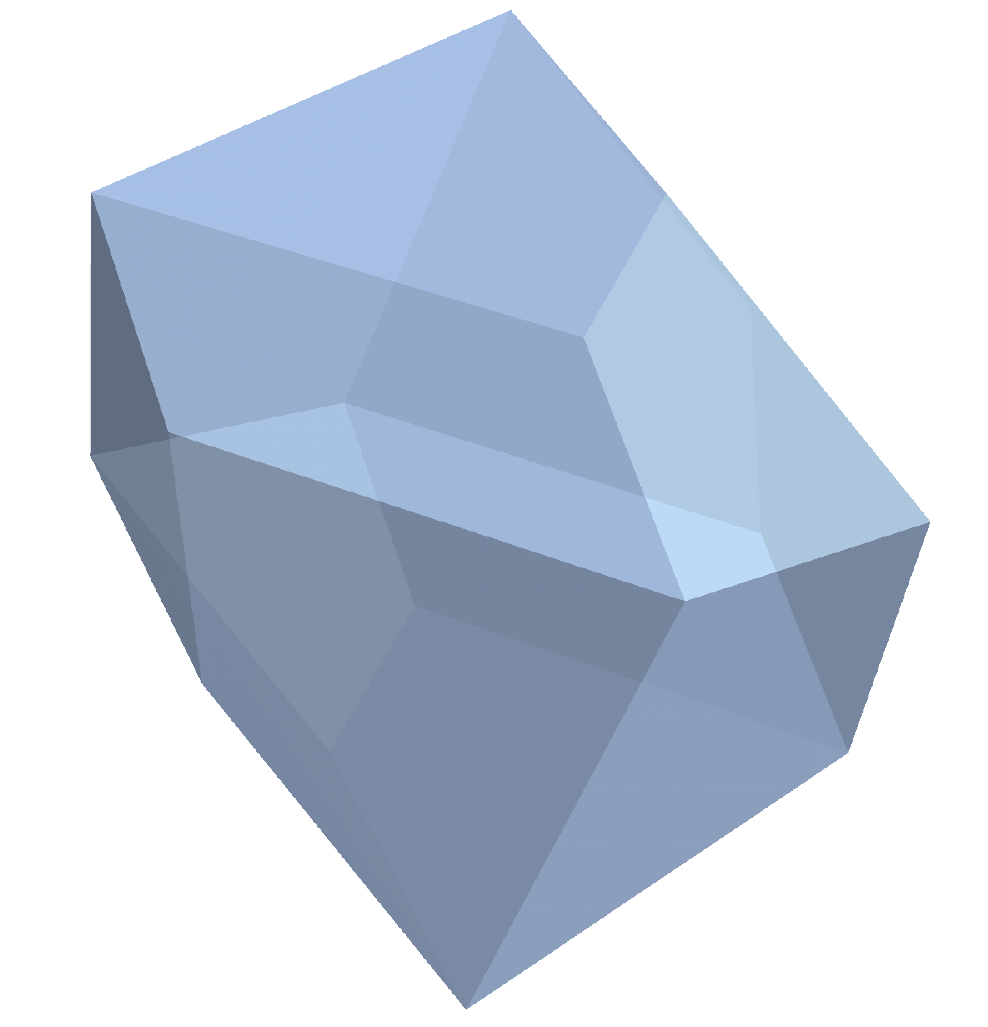}
  %\caption*{$\sig{2}$}
\endminipage\hfill
\minipage{0.32\textwidth}\centering
  \includegraphics[width=5cm]{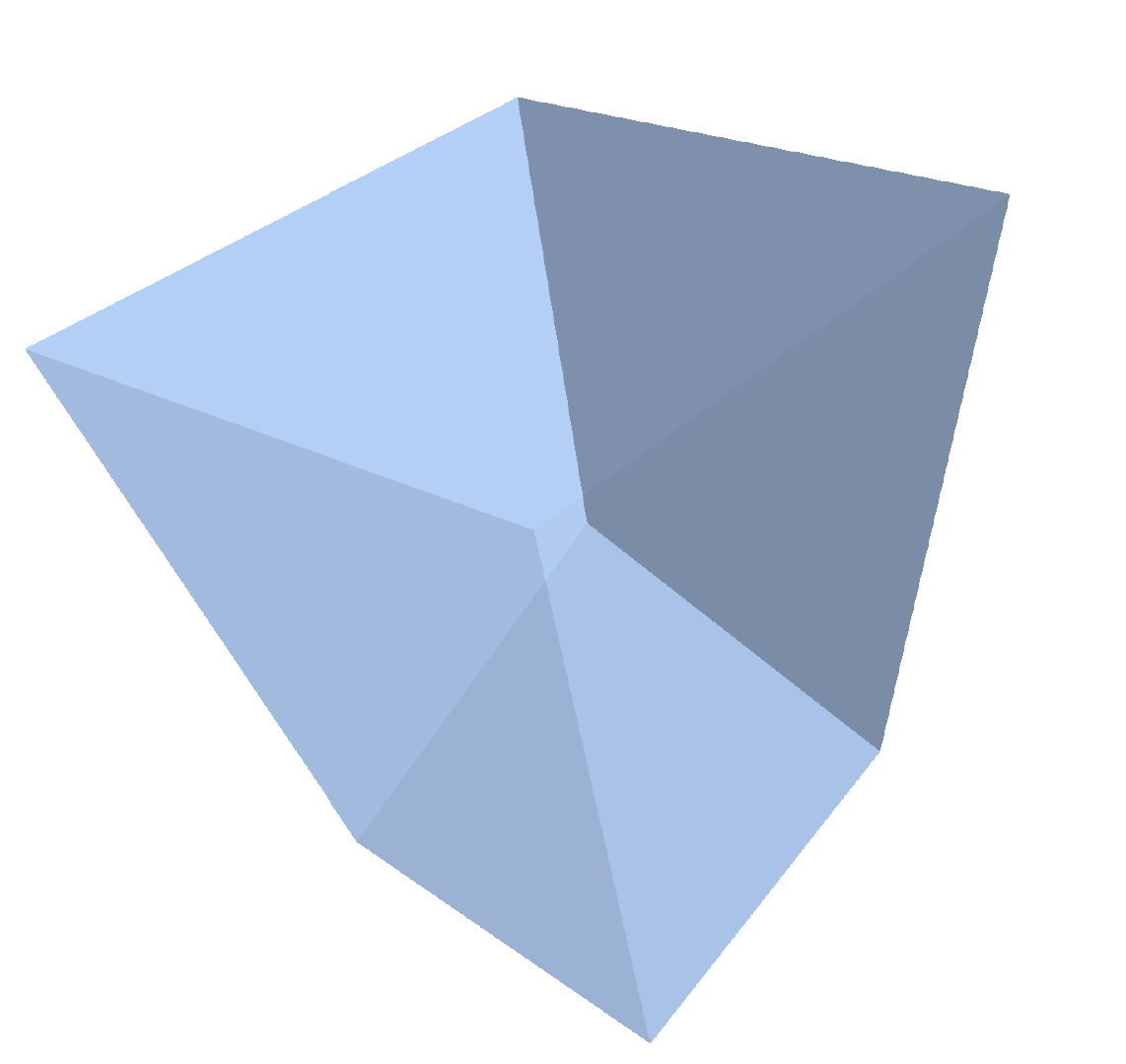}
  %\caption*{$\sig{3}$}
\endminipage
\caption{The three nontrivial higher secondary polytopes ($\sig{1}$, $\sig{2}$, and $\sig{3}$ from left to right) for $d=2, n=5$.}\label{fig:nis5}
\end{figure}

In the generic case, any two tilings of $\zono\subset\R^2$ can be connected by a sequence of flips \cite{kenyon1993tiling}. If two tilings $\T_1$ and $\T_2$ of $\zono$ can be connected by a sequence of flips that do \emph{not} happen at level $k,$ then $\T_1$ and $\T_2$ are called {\it $k$-equivalent}. The vertices of $\sig{k}$ can be interpreted as $k$-equivalence classes of regular tilings, and edges can be interpreted as flips between $k$-equivalent classes of tilings.

From any zonotopal tiling of $\zono$, we can recover a triangulation of $\A$ by $\tau = \{\triangle_B\mid\Pi_{\emptyset,B}\in\T\}$. For higher secondary polytopes, $\tau$ corresponds to a vertex of $\sig{1}$. Galashin, Postnikov, and Williams proved that $\sig{1}$ is a dilation of  the well known GKZ polytope. Although, this is not immediately obvious since the definition of the secondary polytope is based on $A$, the offset vector, while the definition of the secondary polytope is based on $B$, the basis vectors of the tiles.  

For the following theorem and the rest of the paper, we use $\shift$ to mean that two polytopes are equivalent up to translation and dilation.

\begin{theorem}[{\cite[Theorem~2.2(i),~2.2(iv)]{galashin2019higher}}] \label{thm:GPWmain}Let $\A\subset \R^{d-1}$ be a point configuration. Recall $\mathcal{V}\subset\R^d$ is the lift of $\A$ and $\zono$ is the zonotope of $\mathcal{V}$. Then, we have the following.
\begin{enumerate}
    \item $\Sigma_{\A}^{\text{GKZ}}\shift \frac{1}{(d-1)!}\Sigma_{\A,1}$
    \item  $\Sigma_{\A,k}\shift -\Sigma_{\A, n-d-k+1}$.
\end{enumerate}

\end{theorem}

\subsection{Hypertriangulations and lifting hypertriangulations}\label{subsec:hyper}

We can introduce both triangulations of $\A$ and zonotopal tilings of $\zono$ using the language of \emph{(fine) $\pi$-induced subdivisions}, for a linear projection map $\pi$. Without going into detail (see~\cite[Definition~9.1]{ziegler2012lectures} for a complete definition), we mention here that the map $\pi \colon \cube^{\, n} \to \zono$ introduced in Section \ref{sect:fine_tilings} gives rise to zonotopal tilings of $\zono$, while the restriction of $\pi$ onto the simplex $\triangle_{1,n} = \cube^{\, n} \cap \{x_1 + \ldots + x_n = 1\}$ gives rise to triangulations of $\conv \A$. One can similarly consider the restriction $\pi\vert_{\triangle_{k,n}}$ on the $k^{\text{th}}$ hypersimplex, and the corresponding fine $\pi$-induced subdivisions of $\pi(\triangle_{k,n})$ are called \emph{hypertriangulations} in~\cite{olarte2019hypersimplicial}. We observe that the polytope $\pi(\triangle_{k,n})$ being partitioned is the \emph{deleted $k$-sum of $\A$}:
\[
\A^{(k)} = \conv\left\{a_{i_1} + \ldots + a_{i_k} ~\middle\vert~ 1 \le i_1 < \ldots < i_k \le n \right\}.
\]
The hypertriangulations that can be represented as the cross-sections of zonotopal tilings are called \emph{lifting} hypertriangulations. 
\begin{figure}[h]
    \centering
    \includegraphics[width=0.9\textwidth]{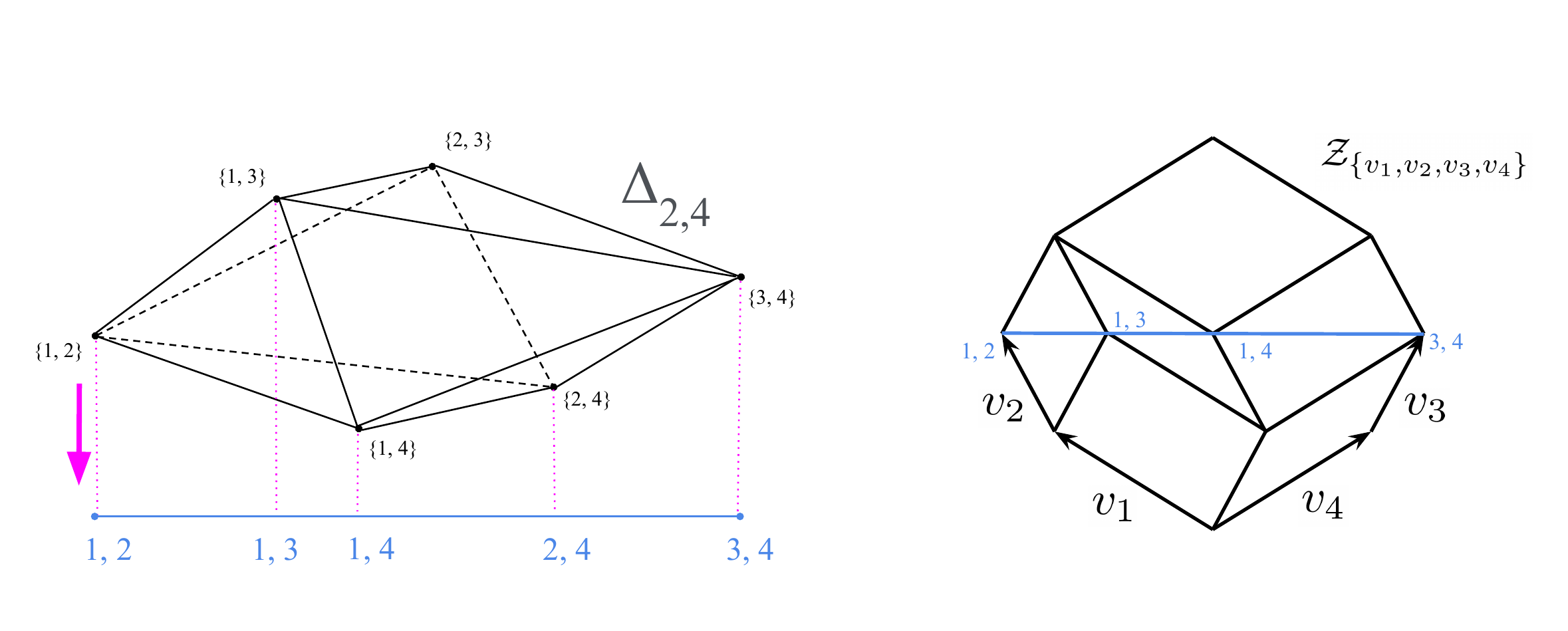}
    \caption{Non-lifting hypertriangulation $12 \to 13 \to 14 \to 24 \to 34$ and lifting hypertriangulation $12 \to 13 \to 14 \to 34$}
    \label{fig:lifting}
\end{figure}

For $d=2$, the $k^{\text{th}}$ level hypertriangulations are \emph{monotone paths on $\triangle_{k,n}$}, that is, chains of $k$-element subsets of $[n]$, starting at $[k]$, ending at $[n]\setminus[n-k]$, and with steps of the form $I \cap \{i\} \to I \cap \{j\}$, $i < j$. Not all hypertriangulations are lifting, as the following example of Galashin \cite[Example~10.3]{postnikov2018positive} shows: the monotone path $12 \to 13 \to 14 \to 24 \to 34$ is not lifting, because the sets $\{1,3\}$ and $\{2,4\}$ are not \emph{strongly separated} (the definitions are given in Section~\ref{sec:hypertriangulations}).

The flips in hypertriangulations correspond to removing/inserting a set from/to the chain, whenever the monotonicity condition is preserved. The flips in lifting hypertriangulations are induced by zonotopal flips, and correspond to removing/inserting a set whenever both monotonicity and separation conditions are preserved.

\subsection{Higher Bruhat order}
A vector configuration is called \textit{cyclic} if for all $v_i\in \mathcal{V}$, $v_i=(a_i^{d-1},...,a_i^2,a_i,1)$ for some real number $a_i$. Let $Z(n,d)$ denote the set of tilings for a cyclic vector configuration $\mathcal{V}$ with $|\mathcal{V}|=n$ and $\mathcal{V}\subset\mathbb{R}^d$.

Ziegler \cite{ziegler1991higher} showed that the flip graph for the set of fine zonotopal tilings of a cyclic configuration of $n$ vectors spanning $\mathbb{R}^d$ is isomorphic to the Hasse diagram of the \textit{higher Bruhat order} $B(n,d)$.  We will not define this poset, but instead will restrict our attention to some of its key properties:

\begin{theorem} (\cite[Theorem 4.1]{ziegler1991higher})\label{thm: bruhat}
\begin{enumerate}
    \item $B(n,d)$ is a graded poset of rank $\binom{n}{d+1}$ with unique minimal and maximal elements which we call $\tmin$ and $\tmax$ respectively.
    \item The flip graph of $Z(n,d)$ forms the Hasse diagram of $B(n,d)$.
    \item There exists a bijection between $Z(n,d+1)$ and the set of commutation classes of length $\binom{n}{d+1}$ flip sequences from $\T_{\min}$ to $\T_{\max}$.
    \item A tile $\Pi_{A,B}$ such that $|A|=k$ in a tiling of $Z(n,d+1)$ corresponds to a flip at level $k+1$ in a maximal chain in $B(n,d)$ corresponding to that tiling.
\end{enumerate}
\end{theorem}

\section{Diameters of higher secondary polytopes for $d=2$}\label{sect:diameter}
Now, we restrict our attention to when $\A$ is generic; that is, we restrict ourselves to the case where $\A$ contains distinct points or, equivalently, when the number of bases of $\R^2$ that can be created from $\mathcal{V}$, the lift of $\A$, is ${[n]\choose 2}$.

%We will first prove a sharp lower bound for the diameter of $\sig{k-1}+\sig{k}$. Then, we prove a sharp lower for the diameter of $\sig{k}.$ 

For notational convenience, we will denote the ``diameter of polytope $P$'' as $\delta(P)$; that is, $\delta(P)$ is the diameter of the $1$-skeleton of $P$, measured using the edge-distance. 

The results of this section are the following two theorems.

\begin{theorem}
\label{thm:diam_k+k-1}
For $d=2$ and for all $k\in[n-d]$, $\delta (\widehat{\Sigma}_k+\widehat{\Sigma}_{k-1})=2k(n-k)-n$.
\end{theorem}

\begin{theorem}
\label{thm:diam_k}
For $d=2$ and for all $k\in[n-d]$, $\delta (\sig{k})=k(n-k-1)$.
\end{theorem}

The four subsections below establish sharp bounds for $\delta (\widehat{\Sigma}_k+\widehat{\Sigma}_{k-1})$ (Propositions~\ref{prop:lbk+k-1}~and~\ref{prop:ubk+k-1}), and sharp bounds for $\delta (\sig{k})$ (Propositions~\ref{prop:lbk}~and~\ref{prop:ubk}), implying Theorems~\ref{thm:diam_k+k-1}~and~\ref{thm:diam_k}.

\subsection{Lower bound for the diameter of $\sig{k}+\sig{k-1}$}\label{lbk+k-1}

In this subsection we will define and develop a potential function that lower bounds the diameter of $\delta(\sig{k}+\sig{k-1})$. For the whole subsection, we fix some zonotope $\mathcal{Z}_{\mathcal{V}}$ and when will refer to tilings of this zonotope without specifying the zonotope $\mathcal{Z}_{\mathcal{V}}$ explicitly. 

We will make use of two particular tilings. The first, $\tmin$, the minimal element in the order 2 Bruhat Order and the second, $\tmax$ the maximal element in the order 2 Bruhat order. 

\begin{figure}[h!]
    \centering
    \includegraphics[width=11cm]{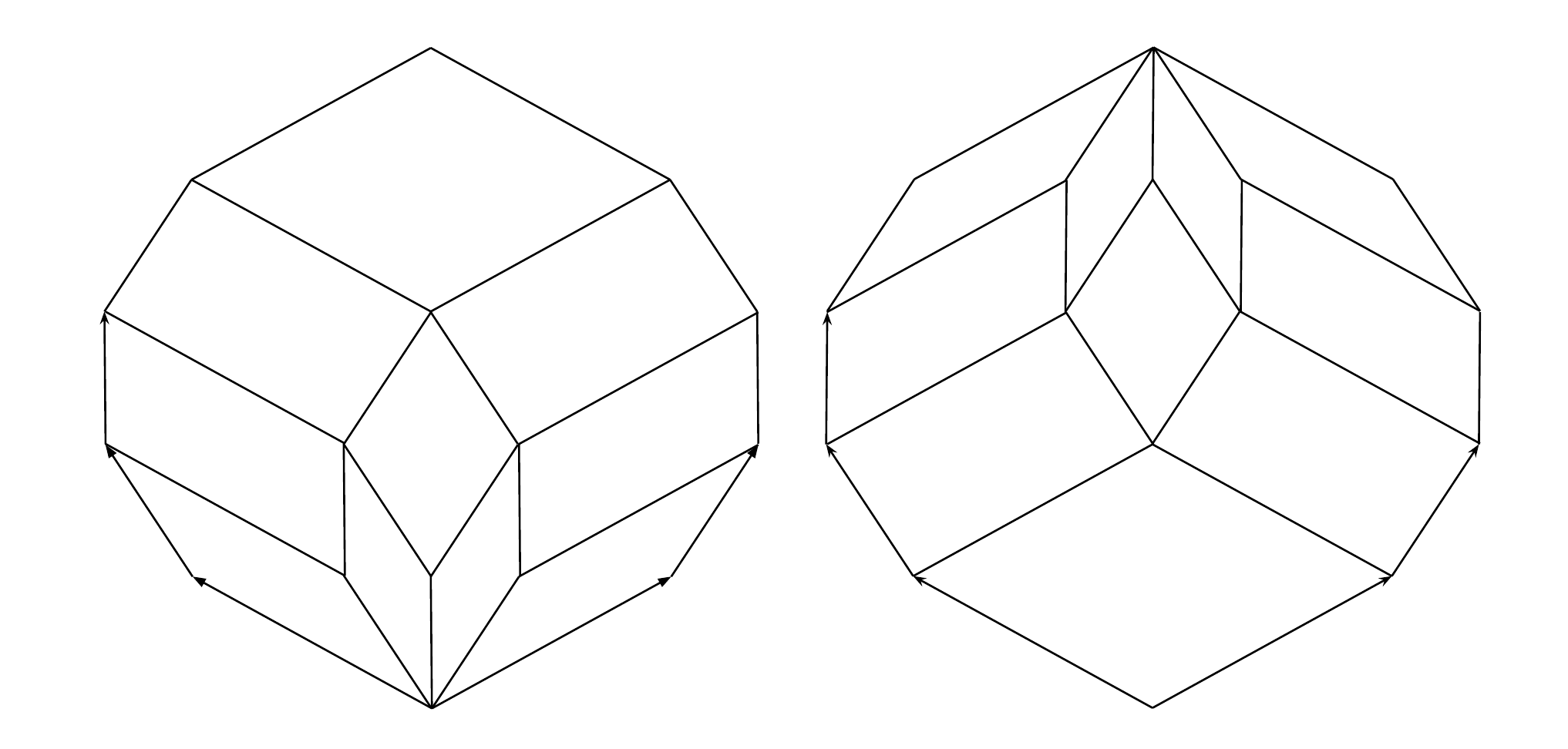}
    \caption{The maximum and minimum tilings for $n=5$.}
    \label{fig:minandmaxtilingsl}
\end{figure}

First, we show that $\tmax$ and $\tmin$ are actually represented by vertices of $\sig{k}$ for all $k$. Indeed,
\begin{lemma}\label{lemma:maxminregularity}
The tilings $\tmin$ and $\tmax$ are regular.
\end{lemma}
\begin{proof}
Without loss of generality, assume that $\A$ is given by $a_i = i$. It is enough to show that $\tmax$ is regular, since if $h\in\R^n$ is the height vector for $\tmax,$ then $-h$ is a height vector that gives $\tmin$. For $n=2$, pick $h=(h_1,h_2)$ so that $h_1<h_2.$ Since the zonotope for $d=2,n=2$ is a parallelogram, this height vector gives the desired zonotope. We construct the minimal tiling on $n$ vectors from the minimal tiling on $n-1$ vectors as follows. Add a vector $(a_n,1)$ to the set $\mathcal{V}$. Assume the heights $h_1, \ldots, h_{n-1}$ are already chosen. Pick the height $h_n \in \R$ large enough so that $$\frac{h_{n-k-2}}{2}+\sum_{i=n-k-1}^{n-1}h_i<\frac{h_n}{2}$$ for all $k$.
Now, at the $k$th level, vertices between $\{n-k-2,...,n-1\}$ and $\{n-k-1,...,n\}$ do not occur. We know this gives us the minimal tiling because it contains (as a subtiling) the minimal tiling of $\mathcal{Z}_{\{v_1,\ldots,v_{n-1}\}}$ by assumption, and since we have only added one tile to every layer, while fixing the tiling of $\mathcal{Z}_{\{v_1,\ldots,v_{n-1}\}}$. 
\end{proof}

In $\tmin$, there are $k$ tiles that intersect the $k\ts{th}$ level for $k\in[n-2]$. In $\tmax$ there are $n-k$ tiles that intersect the $k\ts{th}$ level. Note $\tmin$ and $\tmax$ orient all ${n\choose 3}$ circuits of $\mathcal{V}$ differently, so they are ${n\choose 3}$ flips apart since a flip may only change the orientation of one circuit at a time. Thus, $\tmax$ and $\tmin$ are the maximum flip distance apart and are relatively easy to study; we will use these two tilings to find long paths in $\sig{k}$. We define two subclasses of tiles for a given tiling $\T$.

\begin{definition}\label{def:tiling_orientation}
Let $\T$ be an arbitrary tiling. Let $k\in\{1,...,\lfloor n-2\rfloor\}$. Define 
\begin{align*}
    \T^+_k&=\{B\mid \Pi_{A,B}\in \T, \;|A|\geq k\}\\
    \T^-_k&=\{B\mid \Pi_{A,B}\in \T, \;|A|\leq k-2\},
\end{align*}
which we refer to as the \emph{$\T_k$-positive} and \emph{$\T_k$-negative} tiles, respectively. Notice that $\T^+_k$ depends on the $k$-equivalence class and $\T^-_k$ depends on $(k-1)$-equivalence class. So, the pair $(\T^+_k,\T^-_k)$ depends on the $k$ and $(k-1)$-equivalence class; specifically, tilings that are simultaneously $k$- and $(k-1)$-equivalent have the same positive and negative tiles for a given $k$.
\end{definition}

\begin{definition}\label{def:potential}
Let $\T$ be an arbitrary tiling. For another arbitrary tiling $\T^{\prime}$, let the \emph{level-$k$ potential with respect to $\T$} be defined as $$P_{\T,k}(\T^{\prime})=(|\T^+_k\setminus (\T^{\prime})^+_k|-|(\T^{\prime})^+_k\setminus \T_k^+|) -(|\T_k^-\setminus(\T^{\prime})^-_k|-|(\T^{\prime})^-_k\setminus\T_k^-|).$$
\end{definition}

\begin{lemma}\label{lemma:potential}
Let $\T_1,\T_2$ be two zonotopal tilings that are connected by a single flip. Then, with respect to any zonotopal tiling $\T, |P_{\T,k}(\T_1)-P_{\T,k}(\T_2)| \le 1.$ 
\end{lemma}
\begin{proof}
If the zonotopal flip occurs above level $k$ or below level $k-1$, the potential does not change. The cases of level $k$ and $k-1$ are similar, so we assume that the flip happens on level $k$; then $|(\T_1)_k^-\setminus\T^-_k|=|(\T_2)_k^-\setminus\T^-_k|$ and $|\T^-_k\setminus(\T_1)_k^-|=|\T\setminus (\T_2)_k^-|$. Furthermore, we will only consider flips along circuits so that $|(\T_2)^+_k|-|(\T_1)^+_k|=1$ since flips along circuits so that $|(\T_2)^+_k|-|(\T_1)^+_k|=-1$ are simply the reverse of the flips we consider. And thus, have potential of the same magnitude but opposite sign. So, $P_{\T,k}(\T_1)-P_{\T,k}(\T_2)$ is reduced to $|\T^+_k\setminus(\T_1)_k^+|-|\T^+_k\setminus (\T_2)_k^+|-(|(\T_1)_k^+\setminus\T^+_k|-|(\T_2)_k^+\setminus\T^+_k|).$ 

For the purposes of calculating $|\T^+_k\setminus(\T_1)_k^+|-|\T^+_k\setminus (\T_2)_k^+|$ and $|(\T_1)_k^+\setminus\T^+_k|-|(\T_2)_k^+\setminus\T^+_k|$, we only need to know when a tile is $\T^+_k$ or when it is not. We also note that since the potential is not effected by the particular labels in the set $A$, we can consider circuit combinations up to vertical symmetry. The 6 cases are described below:

\begin{figure}[h!]
    \centering
    \includegraphics[width=13cm]{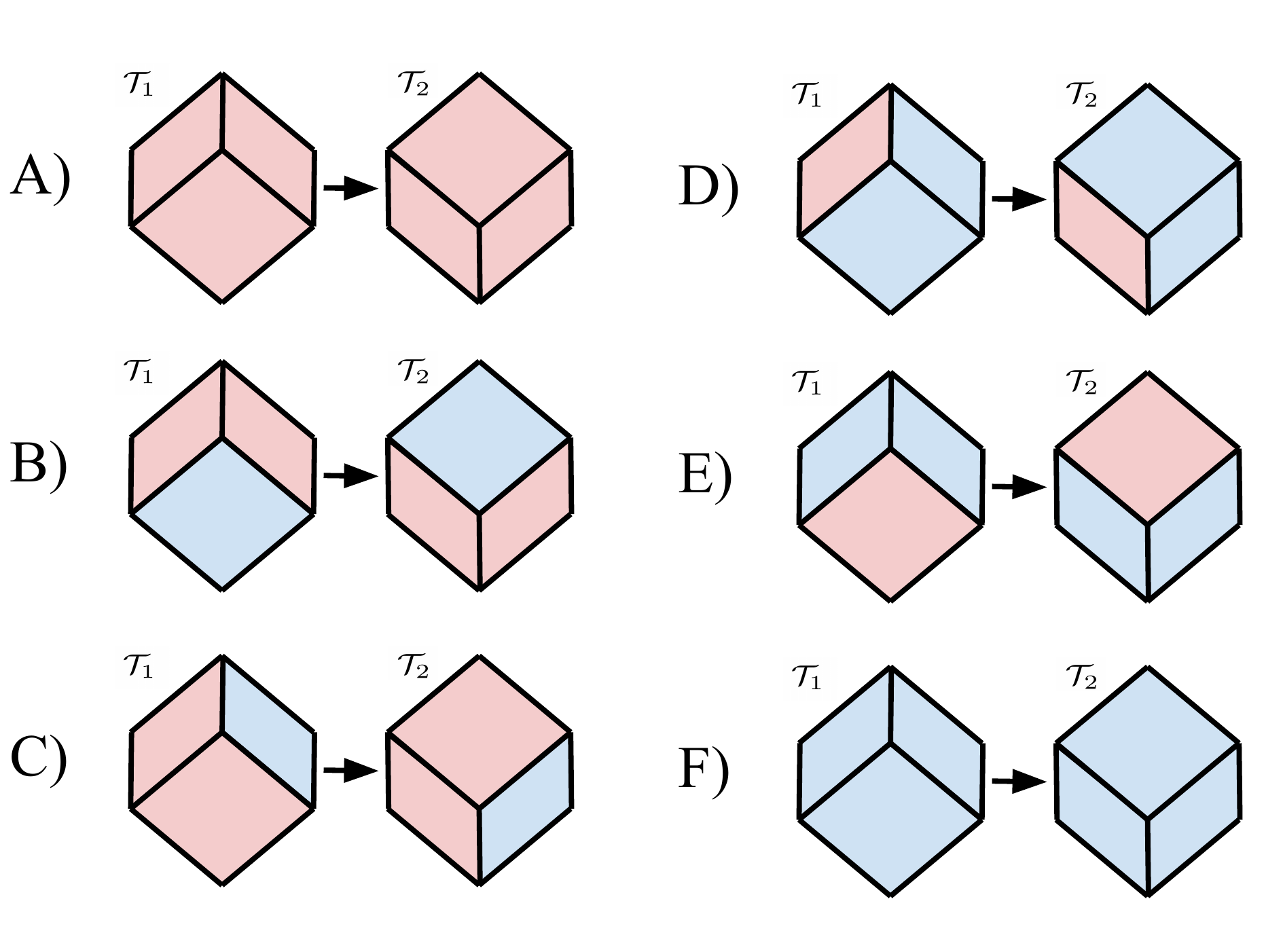}
    \caption{The 6 cases to consider: red tiles represent positively oriented tiles and yellow tiles represent not positively oriented tiles (both with respect to the same tiling $\T$).}
    \label{fig:potenital_cases}
\end{figure}

\begin{enumerate}
    \item[\bf A] Since this flip is only between positive tiles, $|(\T_1)_k^+\setminus\T^+_k|-|(\T_2)_k^+\setminus\T^+_k|=0.$ Furthermore, $|\T^+_k\setminus(\T_1)_k^+|-|\T^+_k\setminus (\T_2)_k^+|=1$ because $|(\T_1)_k^+|>|(\T_2)_k^+|$. Thus,  $|P_{\T}(\T_1)^k-P_{\T,k}(\T_2)|=1$.

    \item[\bf B] A non positive tile with respect to $\T^k$ and $(\T_1)^k$ is positive with respect to $(\T_2)^k$ so $|(\T_1)_k^+\setminus\T^+_k|-|(\T_2)_k^+\setminus\T^+_k|=-1.$ Since the two positively oriented tiles in $\T^k$ and $(\T_1)^k$ are negatively oriented in $(\T_2)_k, |\T^+_k\setminus(\T_1)_k^+|-|\T^+_k\setminus (\T_2)_k^+|=2.$ Thus, $|P_{\T,k}(\T_1)-P_{\T,k}(\T_2)|=1$.

    \item[\bf C]  And, $(\T_1)_k$ orients one tile positively that $\T_k$ does not while $(\T_2)_k$ does not orient any $\T_k-$non positive tiles positively. So, $|(\T_1)_k^+\setminus\T^+_k|-|(\T_2)_k^+\setminus\T^+_k|=-2.$ And, $(\T_1)_k$ and $(\T_2)_k$ orient one $\T_k$-positive tile positively, so $|\T^+_k\setminus(\T_1)_k^+|-|\T^+_k\setminus (\T_2)_k^+|=-1.$ And thus, $|P_{\T,k}(\T_1)-P_{\T,k}(\T_2)|=1$.

    \item[\bf D] There is exactly one $\T_k$-non positive tile oriented positively by $(\T_1)_k$ and $(\T_2)_k$ so $|(\T_1)_k^+\setminus\T^+_k|-|(\T_2)_k^+\setminus\T^+_k|=0.$ However, $(\T_1)_k$ orients one $\T_k-$positive tile positively while $(\T_2)_k$ does not orient any positively, so $|\T^+_k\setminus(\T_1)_k^+|-|\T^+_k\setminus (\T_2)_k^+|=1.$ Thus, $|P_{\T,k}(\T_1)-P_{\T,k}(\T_2)|=1$.

    \item[\bf E] $(\T_1)_k$ orients two $\T$-non positive tiles positively while $(\T_2)_k$ orients a $\T_k-$positive tile positively, so $|(\T_1)_k^+\setminus\T^+_k|-|(\T_2)_k^+\setminus\T^+_k|=2.$ One $\T_k-$positive tile is oriented positively by $(\T_2)_k$ so $|\T^+_k\setminus(\T_1)_k^+|-|\T^+_k\setminus (\T_2)_k^+|=1.$ Thus, $|P_{\T,k}(\T_1)-P_{\T,k}(\T_2)|=1$.

    \item[\bf F] Two $\T_k-$non positive tiles are replaced by 1 $\T$-non positive tile, so $|(\T_1)_k^+\setminus\T^+_k|-|(\T_2)_k^+\setminus\T^+_k|=1$. Since this flip only concerns $\T$-non positive tiles,
$|\T\setminus(\T_1)_k^+|-|\T^+_k\setminus (\T_2)_k^+|=0.$ So, $|P_{\T,k}(\T_1)-P_{\T,k}(\T_2)|=1$.
\end{enumerate}
\end{proof}

\begin{proposition}
\label{prop:lbk+k-1}
If $d=2,$ then for all $n$ and $k\in [n-d], \delta(\widehat{\Sigma}_{k-1}+\widehat{\Sigma}_{k})\geq 2k(n-k)-n$.
\end{proposition}
\begin{proof}
Let $\tmax$ and $\tmin$ be defined as above. Recall that $\tmax$ can be obtained from $\tmin$ by negating all circuits in $\tmax$. Additionally, the path between $\tmin$ and $\tmax$ can be taken to be along all regular tilings by \cite[Lemma~6.5]{galashin2019higher} and Lemma~\ref{lemma:maxminregularity}. We wish to calculate the potential difference $P_{\T,k}$ between $\tmax$ and $\tmin$. 
Let $k\in\{1,...,\lfloor \frac{n}{2}\rfloor\}$. Define:
\begin{align*}
    (\ttmin)^+_k&=\{B\mid \Pi_{A,B}\in \T_{\min}, \;|A|\geq n-k-1\}\\
    (\ttmin)^-_k&=\{B\mid \Pi_{A,B}\in \T_{\min}, \;|A|\leq k-2\}\\
    (\ttmax)^+_k&=\{B\mid \Pi_{A,B}\in \T_{\max}, \;|A|\geq n-k-1\}\\
    (\ttmax)^-_k&=\{B\mid \Pi_{A,B}\in \T_{\max}, \;|A|\leq k-2\}\\
\end{align*}
\begin{figure}[h]
    \centering
    \includegraphics[width=15cm]{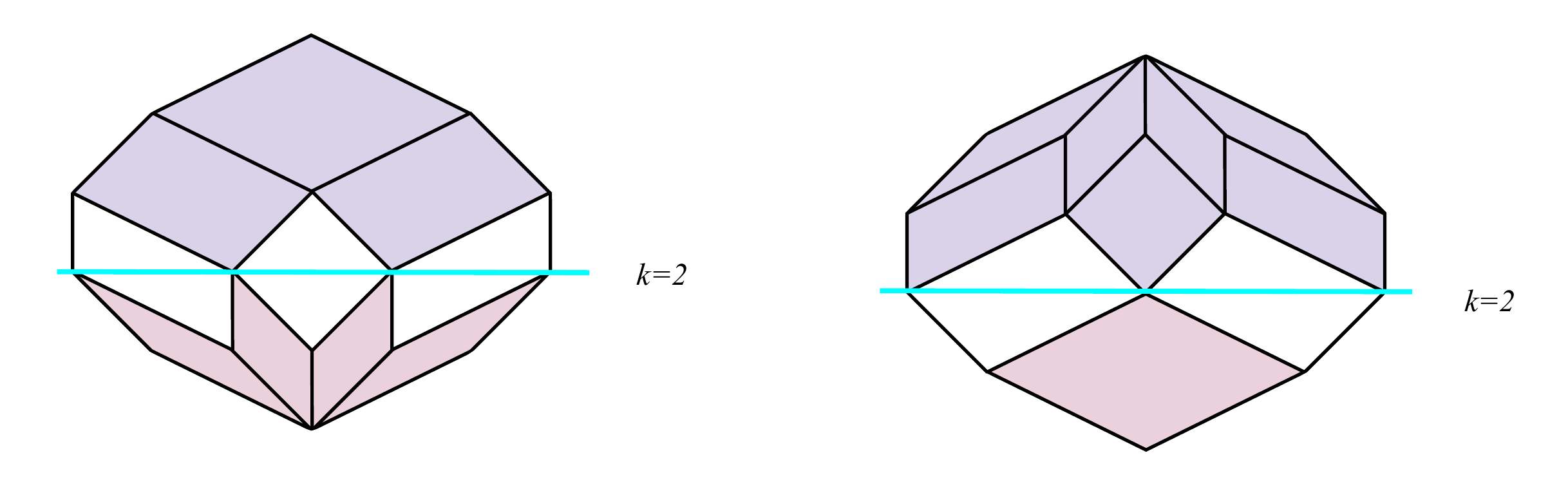}
    \caption{An example of the $(\ttmin)^+_2,(\ttmin)^-_2,(\ttmax)^+_2,(\ttmax)^-_2$ when $n=5$.}
\end{figure}

A calculation of the number of tiles in each row of $\T_{\max}$ and $\T_{\min}$ shows:
\begin{equation*}
\begin{aligned}[c]
|(\ttmin)^+_k| &= \sum_{i=0}^{k-1}n-i-1 \\
&= nk -\frac{k(k+1)}{2}\\
    |(\ttmax)^+_k| &= \sum_{i=0}^{k-1}i+1\\
    &= \frac{k(k+1)}{2}.
\end{aligned}\hspace{1cm}
\begin{aligned}[c]
    |(\ttmin)^-_k| &= \sum_{i=0}^{k-2}i+1 \\
&= \frac{k(k-1)}{2}\\
    |(\ttmax)^-_k| &= \sum_{i=0}^{k-2}n-i-1\\
    &= n(k-1)-\frac{k(k-1)}{2}.
\end{aligned}
\end{equation*}

\noindent Because $k\le \frac{n}{2},(\ttmin)^+_k\cap (\ttmax)^+_k=\emptyset$ and $(\ttmax)^-_k\cap (\ttmax)^-_k=\emptyset$. Thus, 
\begin{align*}
    |P_{(\tmin),k}(\tmax)-P_{(\tmin),k}(\tmin)| &=|(\tmin)^{ +}_k\setminus (\tmax)^+_k|-|(\tmax)^+_k\setminus (\tmin)^{ +}_k|\\
        &\hspace{5mm}-|(\tmin)^{ -}_k\setminus (\tmin)^-_k| +|(\tmin)^-_k\setminus (\tmin)^{-}_k|\\
        &\hspace{5mm}-\Big(|(\tmin)^+_k\setminus (\tmin)^+_k|-|(\tmin)^+_k\setminus (\tmin)^+_k|\\
        &\hspace{5mm}-|(\tmin)^-_k\setminus (\tmin)^-_k|+|(\tmin)^-_k\setminus (\tmin)^-_k|\Big)\\
    &= |(\ttmin)^+_k|-|(\ttmax)^+_k|+|(\ttmin)^-_k|-|(\ttmax)^-_k|\\
    &= -(nk -\frac{k(k+1)}{2})-\frac{k(k+1)}{2}+\frac{k(k-1)}{2}\\&\hspace{5mm}-(n(k-1)-\frac{k(k-1)}{2})\\
    &= k(k+1)+k(k-1)-nk-n(k-1)\\
    &= 2k(n-k)-n.
\end{align*}

 When $d=2,\widehat{\Sigma}_k \shift -\widehat{\Sigma}_{n-k-1}$ and $\widehat{\Sigma}_{k-1} \shift -\widehat{\Sigma}_{n-k}$. So, when $k<n/2, \widehat{\Sigma}_{k-1}+\widehat{\Sigma}_{k}\shift -\Big( \widehat{\Sigma}_{n-k}+\widehat{\Sigma}_{n-k-1}\Big)$. And $\delta (\widehat{\Sigma}_{k-1}+\widehat{\Sigma}_{k})\shift -\Big( \widehat{\Sigma}_{n-k}+\widehat{\Sigma}_{n-k-1}\Big)\geq n(2(n-k)-1)-2(n-k)^2$, so $\widehat{\Sigma}_{k-1}+\widehat{\Sigma}_{k}\geq n(2(n-k)-1)-2(n-k)^2=2k(n-k)-n. $
\end{proof}

\subsection{Upper bound for diameter  of $\sig{k}+\sig{k-1}$}

In this subsection, we show that the lower bound for $\delta(\sig{k}+\sig{k-1})$ obtained in the previous subsection is sharp.

\begin{lemma}
\label{lemma: max chain}
Any regular tiling $\T$ is contained in a path of \emph{regular tilings} of length $\binom{n}{3}$ from $\tmin$ to $\tmax$.
\end{lemma}
\begin{proof}
We can pick $h_{\min}$, $h_{\max}$ and $h$ so that $\T_{h_{\min}}=\tmin$,  $\T_{h_{\max}}=\tmax$, and $\T_h=\T$, and so that the following path in $\mathbb{R}$ passes through at most one hyperplane in $\hv$ at a time: \[
    h(t)=
    \begin{cases}
    (1-t)h_{\min}+th & \text{if } t\in[0,1]\\
    (2-t)h+(t-1)h_{\max} & \text{if } t \in (1,2]
    \end{cases}
\]

Clearly each of the two piecewise-linear pieces of this path can pass through a given hyperplane in $\hv$ at most once.  Assume there exists $t_1\in[0,1],t_2\in(1,2]$, and a circuit $C$ such that $h(t_1)\cdot\alpha(C)=h(t_2)\cdot\alpha(C)=0$.  Expanding this expression, we get the following:
$$
h_{\min}\cdot\alpha(C)=\frac{t_1}{t_1-1}(h\cdot\alpha(C)),
$$
$$
h_{\max}\cdot\alpha(C)=\frac{t_2-1}{t_2}(h\cdot\alpha(C)).
$$
Both $\frac{t_1}{t_1-1}<0$ and $\frac{t_2-1}{t_2}<0$, so $\tmin$ and $\tmax$ have the same circuit orientation for $C$, yielding a contradiction.  Therefore, the path $h(t)$ passes through each of $\binom{n}{3}$ hyperplanes in $\hv$ exactly once, and yields a path of length $\binom{n}{3}$ from $\tmin$ to $\tmax$ which passes through $\T$.
\end{proof}

\begin{proposition}
\label{prop:ubk+k-1}
For all $k\in[n-d]$, $\delta (\widehat{\Sigma}_k+\widehat{\Sigma}_{k-1}) \le 2k(n-k)-n$.
\end{proposition}

Before we proceed with the proof of the above proposition, we note the following: looking at the cross-section of an element of $Z(n,3)$, we get a triangulated $n$-gon.  In this triangulation, we say that a triangle is \emph{white} if it coincides with the bottom cross-section of a tile, and \emph{black} if it coincides with the top cross-section of a tile (see Figure \ref{fig:blackwhite}).
\begin{figure}[h!]\label{fig:blackwhite}
    \centering
    \includegraphics[width=0.4\textwidth]{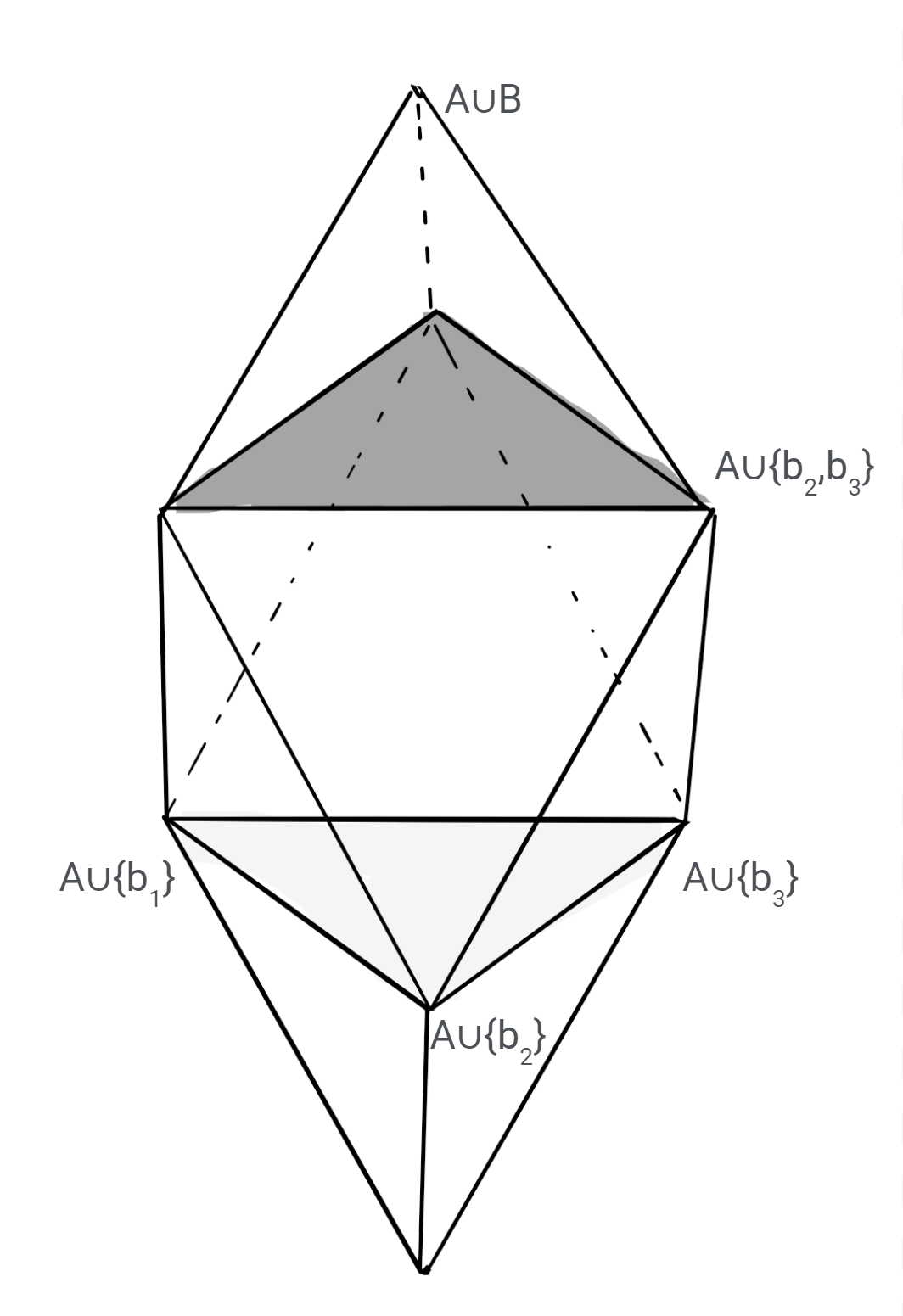}
    \caption{A tile $\Pi_{A,B}$, $B=\{b_1,b_2,b_3\}$, with the upper and lower cross-sections colored accordingly.}
    \label{fig:blackwhite}
\end{figure}

\begin{proof}
Consider a maximal chain in the Bruhat order and its commutation class.  Recall from Theorem \ref{thm: bruhat}  that this commutation class corresponds to an element of $Z(n,3)$, say $\T_{\lift}$.  A flip at level $k$ in this maximal chain corresponds to a tile $\Pi_{A,B}$ with $|A|=k-1$ in $\T_{\lift}$, and similarly a flip at level $k-1$ corresponds to a tile with $|A|=k-2$.  A tile with $|A|=k-1$ or $|A|=k-2$ corresponds to a white or black tile respectively in the $k\ts{th}$ cross-section of $\T_{\lift}$.  The $k\ts{th}$ cross-section of any tiling of the lifted zonotope has $k(n-k)+1$ vertices in total, $n$ of which are exterior vertices.  It is well known that the number of triangles in a triangulation of $p$ points in a plane, with $q$ interior points, is equal to $2p-q-2$, so there are $2k(n-k)-n$ black and white triangles.  Any regular tiling is contained in a maximal chain by Lemma \ref{lemma: max chain}, so given any two regular tilings $\T_{1}$ and $\T_{2}$, we can create a cycle of length $2\binom{n}{3}$ which passes through both tilings as well as $\tmin$ and $\tmax$ which shows that we can relate $\T_{1}$ and $\T_{2}$ by a sequence of flips with at most $2k(n-k)-n$ flips at levels $k$ and $k-1$.

%Combining with the result of Proposition \ref{prop:lbk+k-1}, we conclude that  $\delta (\widehat{\Sigma}_k+\widehat{\Sigma}_{k-1})=2k(n-k)-n$.
\end{proof}

\subsection{Lower bound for diameter of $\sig{k}$}

\begin{proposition}\label{prop:lbk}
For for all $k\in[n-d], \delta (\widehat{\Sigma}_k)\geq k(n-k-1).$ 
\end{proposition}

\begin{proof}
We start by noticing that in the definition of $P_{\T,k}$, considering the first two terms $|\T^+_k\setminus \T^{\prime +}_k|-|\T^{\prime +}_k\setminus \T^+_k|$ change after a flip only if the flip happens at level $k$. Similarly, $|\T^-_k\setminus \T^{\prime -}_k|+|\T^{\prime -}_k\setminus \T^-_k|$ only if the flip happens at level $k-1$. So, we can consider $|\T^+_k\setminus \T^{\prime +}_k|-|\T^{\prime +}_k\setminus \T^+_k|$ to lower bound the number of flips which occur at level $k$ between two tilings. We will call the modified potential $\widetilde{P}_{\T,k}(\T^{\prime})=|\T^+_k\setminus \T^{\prime +}_k|-|\T^{\prime +}_k\setminus \T^+_k|$. 

Fix $k\leq \frac{n}{2}$. We recall the definition of $(\ttmin)^+_k$ and $(\ttmax)^+_k$ from the proof of Proposition \ref{prop:lbk+k-1}. As before, $k\leq \frac{n}{2},(\ttmin)^+_k\cap (\ttmax)^+_k=\emptyset$. Thus, 

\begin{align*}
|\widetilde{P}_{(\tmin),k}(\tmax)-\widetilde{P}_{(\tmin),k}(\tmin)| &=|(\tmin)^{ +}_k\setminus (\tmax)^+_k|-|(\tmax)^+_k\setminus (\tmin)^{ +}_k| \\ &\hspace{5mm}-\Big(|(\tmin)^+_k\setminus (\tmin)^+_k|-|(\tmin)^+_k\setminus (\tmin)^+_k|\\
    &= |(\ttmin)^+_k|-|(\ttmax)^+_k|\\
    &= nk-\frac{k(k+1)}{2}-\frac{k(k+1)}{2}\\
    &= nk-k(k+1)
\end{align*}

When $k>\frac{n}{2},$ recall that $\widehat{\Sigma}_k \shift -\widehat{\Sigma}_{n-k-1}$. So, when $k>n/2, \delta(\widehat{\Sigma}_k)\geq n(n-k-1)-(n-k-1)(n-k)=k(n-k-1).$
\end{proof}

\begin{remark}
It is tempting to think that because $\tmin$ and $\tmax$ are opposites (meaning, all ${n\choose 3}$ of their circuts are oriented differently), an arbitrary pair of opposite tilings would also be distance $k(n-k-1)$ apart in $\sig{k}$. However, consider the two tilings for $n=5$ below. When we restrict to level $k=1$, we see that $\T_1$ and $\T_2$ are 1 only flip move apart in $\sig{1}$, instead of 3, as the lower bound of the diameter would suggest. 
\begin{figure}[h!]
    \centering
    \includegraphics[width=12cm]{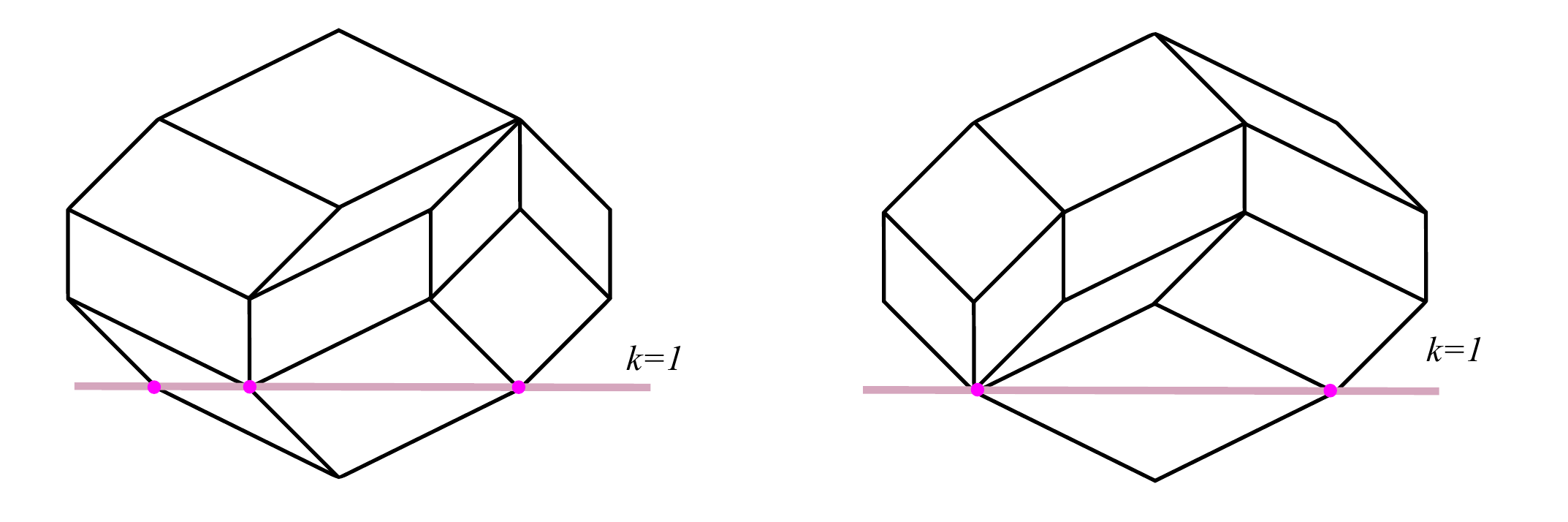}
    \caption{$\T_1$ and $\T_2$ are distance ${n\choose 3}$ apart in the flip graph of zonotopal tilings, but distance 1 apart in $\sig{1}$.}
    \label{fig:nasty}
\end{figure}

\end{remark}
\subsection{Upper bound for diameter of $\sig{k}$}

\begin{proposition}
\label{prop:ubk}
For all $k\in[n-d]$, $\delta (\sig{k})\le k(n-k-1)$.
\end{proposition}
\begin{proof}
This proof is analogous to that of Proposition~\ref{prop:ubk+k-1}.  Now we are interested in the number of white triangles at the $k\ts{th}$ cross-section of $\T_{\lift}$.  Define this quantity to be $W(k)$. Since a black triangle at the $k^{th}$ cross-section corresponds to a white triangle at the $(k-1)\ts{th}$ cross-section, 

\begin{align*}
    W(k)&=2k(n-k)-n-W(k-1)\\
    &=\sum_{i-1}^k(-1)^{k+i}\left[2i(n-i)-n\right]\\
    &=nk-k(k+1).
\end{align*}

\end{proof}

\section{Diameters of flip graphs of lifting hypertriangulations and reduced lifting hypertriangulations for $d=2$}
\label{sec:hypertriangulations}

%\textcolor{green}{I've just noticed that the flip diameter for (not necessarily lifting) hypertriangulations is super easy to compute, why won't we add this here or in the background! it's just $k$, for $(k,n)$-monotone paths, when $k \le n/2$, and $n-k$ otherwise}

As mentioned in Section \ref{subsec:hyper}, lifting hypertriangulations have an additional requirement over hypertriangulations in that their underlying sets must be \emph{strongly separated}:

\begin{definition}
We say two sets $S_1$ and $S_2$ are \emph{strongly separated} if $\max(S_1\setminus S_2) < \min(S_2\setminus S_1)$ or $\max(S_2\setminus S_1)<\min( S_1\setminus S_2)$.  A collection of subsets is \emph{strongly separated} if the subsets are pairwise strongly separated.
\end{definition}

We can also define \emph{reduced lifting hypertriangulations}:

\begin{definition}
A \textit{reduced} lifting hypertriangulation at level $k$ is a lifting hypertriangulation at level $k$ such that for any three consecutive vertices $A_1,A_2,A_3$, $|A_1\cap A_2\cap A_3|=k-2$.
\end{definition}

As the word ``reduced'' suggests, one can throw away a few vertices from any lifting hypertriangulation to get the corresponding reduced lifting hypertriangulation.

\begin{figure}[h]
    \centering
    \includegraphics[width=0.3\textwidth]{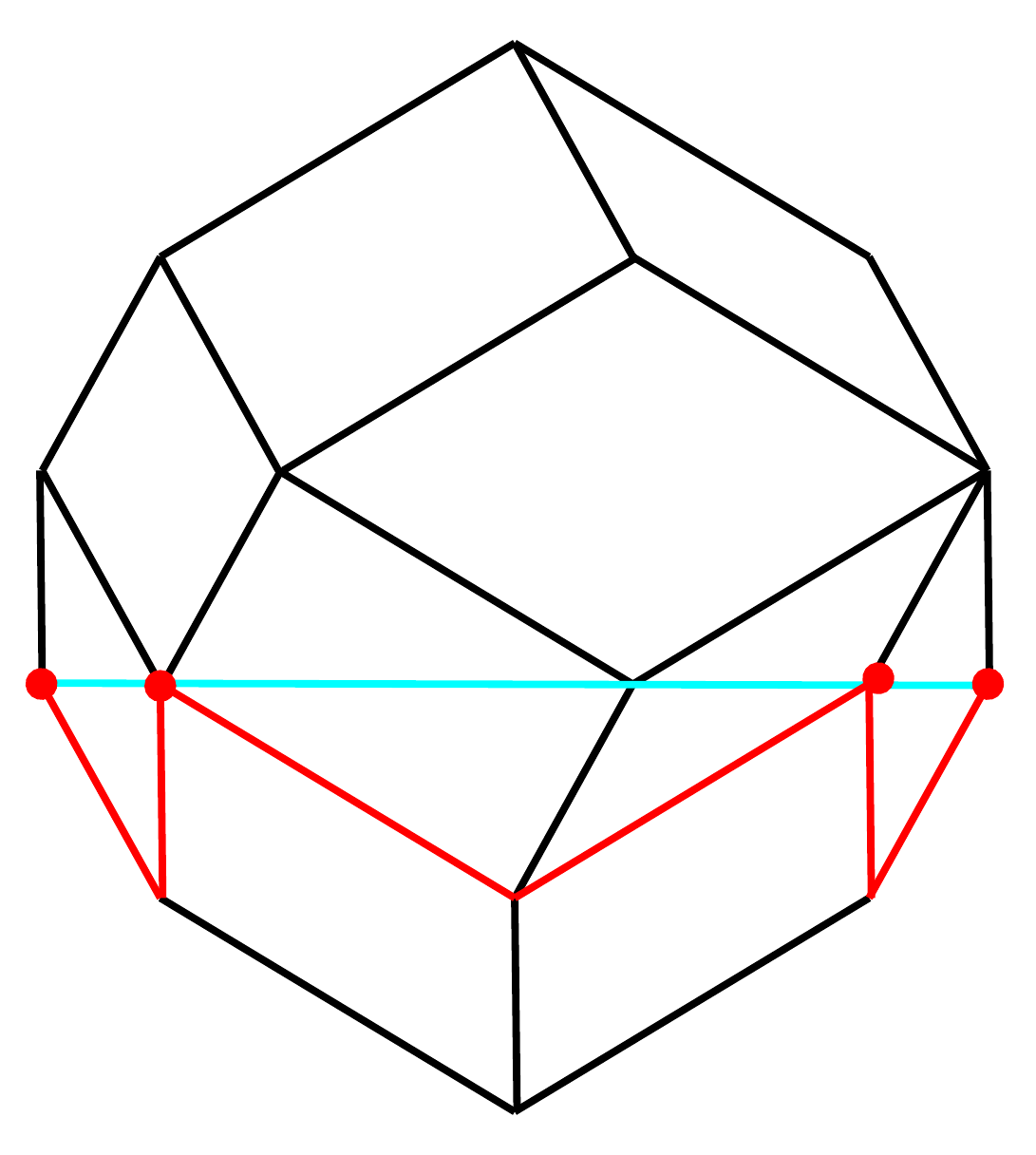}
    \caption{A tiling with of $\zono$ when $\A=\{-2,-1,0,1,2\}$.  Here the lifting triangulation at level 2 has vertices \{1,2\},\{1,3\},\{3,4\},\{3,5\},\{4,5\}, but the corresponding reduced hypertriangulation excludes vertex \{3,5\}. }
    \label{hyperflip}
\end{figure}

As mentioned in Section \ref{subsec:hyper}, flips in lifting hypertriangulations are those induced by flips in zonotopal tilings.  Say $M$ is a flip in a hypertriangulation corresponding to the $k\ts{th}$ cross-section of a tiling $\T$.  We say $M$ is an \emph{upper flip} if it is induced by a flip at level $k$ in $\T$.  Similarly, we say $M$ is a \emph{lower flip} if it is induced by a flip at level $k-1$ in $\T$.

In this section, we will show that lifting hypertriangulations and reduced lifting hypertriangulations, and the flips between them relate to equivalence classes of fine zonotopal tilings and flips between these equivalence classes. Using this relationship, we calculate the diameter of the flip graphs of lifting hypertriangulations and reduced lifting triangulations.

\subsection{Diameter of the flip graph of lifting hypertriangulations}

\begin{proposition}\label{prop:hyper}
Lifting hypertriangulations at level $k$ and the flips between them correspond to classes of simultaneous $k$- and $(k-1)$-equivalence and the flips between these classes, respectively.
\end{proposition}

\begin{lemma}\label{lemma:path}
Say there is a possible upper flip $M$ in a lifting hypertriangulation coinciding with the $k\ts{th}$ cross-section of tiling $\T$ (this hypertriangular flip might not be available as a zonotopal flip in $\T$). Then there exists a finite sequence of zonotopal flips $(F_1,....,F_m)$ starting at $\T$ such that the $\ell\ts{th}$ lifting hypertriangulation is unchanged by the first $m-1$ flips for all $\ell\leq k$, but the upper flip $M$ occurs on the last zonotopal flip $F_m$.
\end{lemma}
\begin{proof}
Say $k=n-3$.  There is only one possible position of tiles at the $(n-2)\ts{th}$ cross-section, so this case is trivially true.

We induct downwards.  Let $V_k(\T)$ be the set of vertices in the $k\ts{th}$ lifting hypertriangulation of $\T$.  Say there is a vertex $S\in V_k(\T_1)$ but  $S\notin V_k(\T_2)$.
 If there is a flip from $\T_1$ that removes $S$ from $V_k(\T_1)$ then perform this flip and we are done.  Otherwise, there is at least one vertex that should be removed from $V_{k+1}(\T_1)$ to make the zonotopal flip available; let $S'$ be one of those vertices.  Inductively, there exists a sequence of flips from $\T_1$ that will remove $S'$ where each flip (except the last) leaves the $\ell\ts{th}$ hypertriangulation the same for all $\ell\leq k+1$.  The last flip will also leave the $k\ts{th}$ cross-section the same, so after repeating this procedure for all obstructions $S'$, we can perform the flip $F_m$ in our tiling which induces the flip $M$ in the $k\ts{th}$ cross-section.

\end{proof}

\begin{proof}[Proof of Proposition~\ref{prop:hyper}]
Say tilings $\T_1$ and $\T_2$ are $k$- and $(k-1)$-equivalent.  Then $(\T_1)^+_k=(\T_2)^+_k$ and $(\T_1)^-_k=(\T_2)^-_k$, as defined in Definition \ref{def:tiling_orientation}.  Then we can completely reconstruct the set of $k\ts{th}$ cross-section of the tilings, so their lifting hypertriangulation at level $k$ must be the same.

Now assume the lifting hypertriangulation at level $k$ is the same for $\T_1$ and $\T_2$.  It suffices to show that we can use a series of flips not at level $k$ or $k-1$ to connect $\T_1$ to another tiling $\T_1'$ such that the $(k+1)\ts{th}$ lifting hypertriangulation of $\T_1'$ matches that of $\T_2$; by symmetry, we can also match the $(k-1)\ts{th}$ lifting hypertriangulations, and then we can proceed inductively so that we get a sequence of flips connecting $\T_1$ and $\T_2$.   This is true by Lemma \ref{lemma:path}.

The statement that the flips between these lifting hypertriangulations correspond to flips at levels $k$ and $(k-1)$ is now easy to check because one of these flips necessarily adds or removes a vertex to the hypertriangulation:

\begin{figure}[h]
    \centering
    \includegraphics[width=0.3\textwidth]{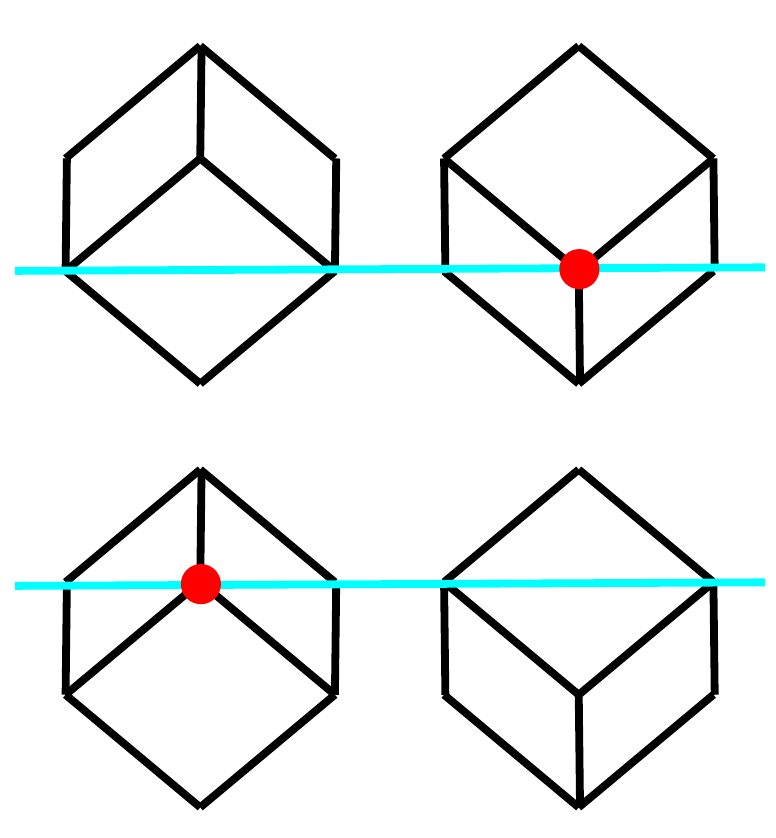}
    \caption{An \emph{upper} flip at level $k$ (top) and \emph{lower} flip at level $k-1$ (bottom).  The vertex that is introduced to or removed from the $k\ts{th}$ lifting hypertriangulation in each case is highlighted in red}
\end{figure}
\end{proof}

The problem of finding the flip diameter of these hypertriangulations is now reduced to finding the least number of total flips at level $k$ and $k-1$ required to connect any two fine zonotopal tilings (not necessarily regular). The lower bound from Proposition \ref{prop:lbk+k-1} clearly still holds, but here we slightly modify the proof of the upper bound to account for possible irregular tilings.
\begin{theorem}
\label{thm:up-fgk+k-1}
The diameter of the flip graph for lifting hypertriangulations at level $k$ is $2k(n-k)-n$.
\end{theorem}

\begin{proof}
The proof of Theorem~\ref{thm:up-fgk+k-1} is completely analogous to that of Theorem~\ref{thm:diam_k+k-1}.  The only difference is that any tiling is contained in a maximal chain in $Z(n,2)$ which may or may not contain irregular tilings.
\end{proof}

\subsection{Diameter of the flip graph of reduced lifting hypertriangulations}\label{subsec:reduced}

\begin{proposition}
The reduced lifting hypertriangulations at level $k+1$ and the flips between them correspond to $k$-equivalence classes of fine zonotopal tilings and the flips between them respectively.
\end{proposition}
\begin{proof}
The proof that the reduced lifting hypertriangulations at level $k+1$ correspond to $k$-equivalence classes of tilings is completely analagous to the proof of Prop. \ref{prop:hyper}.  

Then it remains to check that a flip at level $k$ introduces or deletes a point in the reduced hypertriangulation at level $k+1$.  Referring to the bottom row of Figure \ref{hyperflip}, we see that a flip at level $k$ would introduce/remove at least one vertex to the lifting hypertriangulation and that this vertex would appear in the reduced lifting hypertriangulation (note the reduced hypertriangulation is not affected by an upper flip at level $k+1$).
\end{proof}

\begin{theorem}
\label{thm:ub-fgk}
The diameter of the flip graph for reduced lifting triangulations at level $k+1$ is $k(n-k-1)$.
\end{theorem}

\begin{proof}
The proof of Theorem \ref{thm:ub-fgk} is completely analogous to that of Theorem \ref{thm:diam_k}.  The only difference is that any tiling is contained in a maximal chain in $Z(n,2)$ which may or may not contain irregular tilings.  The lower bound is maintained from Proposition~\ref{prop:lbk} since the graph of all---regular and irregular---tilings for is flip connected in this case.
\end{proof}

\bibliographystyle{alpha}
\bibliography{bibliography}
\iffalse

Department of Mathematics, Massachusetts Instittute of Technology, Cambridge MA\\
{\bf\it Email Address}, E. Bullock: \href{mailto:me@example.com}{edb22@mit.edu}\\
\emph{Email Address,}  K. Gravel: \href{mailto:me@example.com}{kgravel@mit.edu}
\fi

\end{document}